\documentclass[12pt,twoside]{amsart} 
\date{\today, version 1.00}

\newcommand{\C}{{\mathbb{C}}}
\newcommand{\Q}{{\mathbb{Q}}}
\newcommand{\Qq}{{\mathbb{Q}}}
\newcommand{\R}{{\mathbb{R}}}

\newcommand{\OO}{{\mathcal{O}}}
\newcommand{\Supp}[0]{{\operatorname{Supp}}}

 \newcommand{\bir}{\dashrightarrow}
 \newcommand{\N}{\mathbb{N}}

\DeclareMathOperator{\vol}{vol}
\DeclareMathOperator{\Ivol}{Ivol}

\usepackage{pxfonts}

\usepackage{latexsym}
\usepackage{amssymb}
\usepackage{amsthm}
\usepackage{amscd}
\usepackage{enumerate} 
\usepackage{amssymb} 
\usepackage{mathrsfs}          
\usepackage[all]{xy}
\usepackage{color}

\newtheorem{thm}{Theorem}[section]

\newtheorem{lem}[thm]{Lemma}

\newtheorem{exa}[thm]{Example}

\newtheorem{ques}[thm]{Question}

\newtheorem{claim}[thm]{Claim}
 
\theoremstyle{definition}

\title{\large V\MakeLowercase{ariations of generalised pairs}}

\author{C\MakeLowercase{aucher} B\MakeLowercase{irkar and} 
C\MakeLowercase{hristopher} D. H\MakeLowercase{acon}}

\begin{document}
\bibliographystyle{amsalpha+}

 \maketitle
 
\abstractname{: 
In this paper we investigate various properties of
generalised pairs in families, especially boundedness of several kinds.
We show that many statements for usual pairs do not hold for generalised pairs. In particular, we construct an unexpected counter-example to boundedness of generalised lc models with fixed appropriate invariants. We also show that the DCC of Iitaka volumes and existence of nef reduction maps fail in families of generalised pairs.

 In a positive direction we show boundedness of bases of log Calabi-Yau fibrations with their  induced generalised pair structure under natural assumptions. Roughly speaking we prove this boundedness for fibrations whose general fibres belong to a bounded family and whose Iitaka volume is fixed. 
}
 
\tableofcontents 

\section{Introduction}

In this paper we work over an algebraically closed field $k$ of characteristic zero.\\

\textbf{Generalised pairs.}
In recent years the theory of generalised pairs has been instrumental in many advances in higher dimensional algebraic geometry. Roughly speaking a generalised pair consists of a pair $(X,B)$ and a nef divisor 
on some birational model of $X$. One of the most natural contexts in which generalised pairs arise is that of varieties that admit a log Calabi-Yau fibration. Calabi-Yau and Fano varieties, their fibre spaces, and their local counterparts play a central role in modern mathematics, and they are all special cases of log Calabi-Yau fibrations. 

A log Calabi-Yau fibration consists of a pair $(V,\Delta)$ with log canonical (lc) singularities together with a contraction $f\colon V\to X$, i.e. a projective morphism with connected fibres, such that $K_V+\Delta\sim_\Q 0$ over $X$ (note that here we allow the extreme cases when $f$ is birational, identity, or even constant). See \cite{B18} for a systematic treatment. The canonical bundle formula, also known as adjunction, says that we can write 
$$
K_V+\Delta\sim_\Q f^*(K_X+B+M)
$$
where $(X,B+M)$ naturally has the structure of a generalised pair: $B$ measures the singularities of the fibres and $M$ measures the variation of the fibres in their moduli space. The geometry of $(V,\Delta)$ is often studied in terms of the geometry of the fibres of $(V,\Delta)\to X$ and of the base $(X,B+M)$; see for example,  \cite{BZ16} and \cite{B19}.

Standard birational geometry conjectures imply that all varieties of intermediate log Kodaira dimension admit such a log Calabi-Yau fibration, after some birational transformations. This constitutes a very large class of varieties, and understanding their geometry is a top priority for higher dimensional geometry and related areas. 

Another natural context in which generalised pairs appear is that of projective pairs $(X,B)$ such that $-(K_X+B)$ is nef. Letting $M:=-(K_X+B)$ we obtain a generalised pair $(X,B+M)$.  Many of the examples studied in this paper are constructed out of such pairs.   

The purpose of this text is to investigate various properties of  
generalised pairs in families, especially boundedness of several kinds, and then present applications to usual pairs and varieties. Generalised pairs behave similarly to the usual pairs in many ways however, as we will see in this paper,
there are several crucial differences and unexpected phenomena. Many of the questions we are interested in can be 
reduced to questions in the context of a fixed family as in the following setup.

\subsection{Setup}\label{setup}
\emph{
Assume $(X,B)$ is a pair where $B$ is a $\Q$-boundary, $M$ is a $\Q$-divisor, and $X\to T$ is a contraction onto a smooth variety. Assume that 
\begin{itemize}
\item $(X,\Supp B\cup \Supp M)$ is relatively log smooth over $T$,
\item for a dense set of closed points $T'\subset T$, $M_t$ is nef for any $t\in T'$ where 
$X_t$ is the fibre over $t$, $B_t=B|_{X_t}$ and $M_t=M|_{X_t}$.
\end{itemize}} 

For each $t\in T'$, $(X_t,B_t+M_t)$ is a generalised pair with nef part $M_t$.  
Note that this does not imply that $M$ is nef over $T$ and so in general $(X,B+M)$ is not a generalised pair.
However, it does follow that $M_\eta$ is nef, where $\eta$ is the generic point of $T$. To see this note that if $A$ is an ample$/T$ divisor and $M_t$ is nef, 
then $M_t+\epsilon A_t$ is ample for any $\epsilon>0$. Since ampleness is an open condition, it follows that $M_\eta+\epsilon A_\eta$ 
is ample for any $\epsilon>0$. In particular, $(X_\eta,B_\eta+M_\eta)$ is a generalised pair. 

Imposing various kinds of conditions on the generalised pairs $(X_t,B_t+M_t)$, for $t\in T'$, 
we would like to see if the same conditions are satisfied for $(X_\eta,B_\eta+M_\eta)$ or 
possibly for all fibers $X_t$ over an open subset $U\subset T$.
This leads to many interesting questions that are motivated by important applications and problems about usual pairs.

Under the assumptions of \ref{setup}, 
assume that $(X_t,B_t+M_t)$ has a (good) minimal model for every $t\in T'$. 
\begin{enumerate}
\item Does $(X_\eta,B_\eta+M_\eta)$ have a (good) minimal model? 

\item Is $R(K_{X_\eta}+B_\eta+M_\eta)$ finitely generated over $k(T)$?

\item Can we construct a minimal model for $(X_t,B_t+M_t)$, for each $t\in T'$, 
so that the minimal models form a bounded family?

\item Does $(X_t,B_t+M_t)$ have a minimal model for every $t$ in some open neighbourhood of $T'$? 

\item Can we run an MMP on $K_X+B+M$ over some open neighbourhood of $T'$ so 
that it terminates with a minimal model of $(X,B+M)$ over that neighbourhood?

\item Does the Iitaka volume ${\rm Ivol}(K_{X_t}+B_t+M_t)$ satisfy the DCC, in particular, is there 
a lower bound for the Iitaka volume? 

\item If $K_{X}+B+M$ is nef over $T$, does there exist a relative nef reduction map 
for $K_X+B+M$ over $T$? 
\end{enumerate}

It is easy to see that the answer to (1) is negative as shown by Example \ref{e-a} below. We will also see in \S 5 that the answer to (2) is negative. This is in stark contrast to the case of usual pairs, that is, when $M=0$, 
because in this case even if $(X_t,B_t)$ has a good minimal model for only one $t$, 
then $(X_\eta,B_\eta)$ has a good minimal model and $R(K_{X_\eta}+B_\eta)$ is finitely generated;  in fact, $(X,B)$ has a good minimal model over $T$ \cite{HMX18}.

\begin{exa}\label{e-a}
\emph{
Assume that $E$ is an elliptic curve over the complex numbers and $X=E\times E\to T=E$ and $M$ is the
Poincar\'e line bundle and $B=0$. Let $T'$ be the set of closed points $t$ such that 
$M_t$ is torsion. Then $K_{X_t}+B_t+M_t$ is torsion for each $t\in T'$, so 
$(X_t,B_t+M_t)$ is already a good minimal model for such $t$. However, 
$K_{X_t}+B_t+M_t$ is not torsion for any $t\notin T'$. In particular, 
$K_{X_\eta}+B_\eta+M_\eta$ is not torsion, so $(X_\eta,B_\eta+M_\eta)$ does not 
have a good minimal model. }
\end{exa}

We will in fact construct counter-examples to Questions (1) and (2) when $K_{X_t}+B_t+M_t$ is big, which is much more subtle (see Section 5). 
In particular, from this we deduce that the generalised lc models of a bounded family of 
generalised pairs do not always form a bounded family, hence we get a 
counter-example to a previously expected conjecture. 

More precisely, 
given $d,p\in \mathbb{N}$ and $v\in \Q^{>0}$, it was expected that the set 
$\mathcal{F}_{glc}(d,p,v)$ of projective 
generalised lc pairs $(X,B+M)$ with nef part $M'$ (on some birational model of $X$) 
such that $\dim X=d$, $pM'$ is Cartier, and $K_X+B+M$ is ample with 
volume $v$, forms a bounded family. 
This was first studied in \cite{Fil18} for surfaces. Boundedness has been confirmed in the generalised klt case more recently \cite{B21} but we will show that it fails in the generalised lc case (see Section 5). In fact this paper started out of trying to understand boundedness of $\mathcal{F}_{glc}(d,p,v)$. 

Although boundedness of $\mathcal{F}_{glc}(d,p,v)$ fails in full generality, we still expect that some weaker variants are true. 

\begin{ques}\label{q-bnd-gldt-model} 
Can we choose a generalised dlt model $(Y,B_Y+M_Y)$ for each $(X,B+M)\in \mathcal{F}_{glc}(d,p,v)$ 
so that the $(Y,B_Y,M_Y)$ form a bounded family? 
\end{ques}

We have a positive answer to this question in dimension $\le 3$ (not included in this paper) and we expect a positive answer in any dimension. Thus we expect a positive answer to Question (3). 

We will see that the answer to Question (6) is also negative already in dimension 2 (see Section 3). 
In a similar context we will show that the answer to Question (7) is also negative (see Section 4).

\bigskip

\textbf{Generalised pairs on bases of fibrations.}
Our next goal is to investigate boundedness of bases of log Calabi-Yau fibrations, more precisely boundedness of the generalised pairs 
on the bases of log Calabi-Yau fibrations given by the canonical bundle formula. 
The following theorem essentially says that considering lc log Calabi-Yau fibrations  where the general fibres are bounded and the Iitaka volume is fixed we get boundedness of their bases.

\begin{thm}\label{t-bnd-base-adjunction} 
Let $d\in \mathbb N$, $\Phi\subset \Q^{\ge 0}$ be a DCC set, and $u,v\in \Q^{>0}$. 
Consider the set of projective pairs $(X,B)$ and $\Q$-divisors $A\ge 0$ on $X$ such that 
\begin{itemize}
\item $(X,B)$ is lc of dimension $d$,
\item $K_X+B$ is semi-ample defining a contraction $f\colon X\to Z$,
\item over the generic point $\eta_Z$: $A$ is ample and contains no non-klt centre of $(X,B)$,
\item the coefficients of $B$ and the horizontal coefficients of $A$ are in $\Phi$, 
\item $\vol(A|_F)=u$ for general fibres $F$ of $f$, and 
\item ${\rm Ivol}(K_X+B)=v$.
\end{itemize}
Then we can write the adjunction formula 
$$
K_X+B\sim_\Q f^*(K_Z+B_Z+M_Z)
$$ 
such that the corresponding set of generalized pairs $(Z,B_Z+M_Z)$ forms a bounded family. In particular there exists a very ample divisor $H$ on $Z$ such that $H^{\dim Z}$, $B_Z\cdot H^{\dim Z-1}$ and $M_Z\cdot H^{\dim Z-1}$ are bounded by a fixed constant.
\end{thm}

The theorem is very natural from the point of view of polarised varieties and their moduli \cite{B20}\cite{B21}. Indeed it is a crucial step towards construction of moduli spaces that will be discussed elsewhere. 

The hypothesis of Theorem \ref{t-bnd-base-adjunction} imply that the general fibres are bounded in a uniform sense, by \cite[Corollary 1.8]{B20}. 
Moreover, the assumptions also imply that 
$$
(Z,B_Z+M_Z)\in \mathcal{F}_{glc}(e,p,v)
$$
where $e\le d$ and $p$ is fixed. Although $\mathcal{F}_{glc}(e,p,v)$ is not a bounded family in general but the theorem says that the subset given by bases of log Calabi-Yau fibrations forms a bounded family. When $(X,B)$ is klt this boundedness follows from \cite{B21} but the lc case needs a lot more work as we will see in this paper.

Although the bases $(Z,B_Z+M_Z)$ form a bounded family, the pairs $(X,B)$ do not need to form a bounded family. Indeed here is a simple example. Consider a Hirzebruch surface, i.e. $X={\rm Proj} (\OO_Z\oplus \OO_Z(n))$ over $Z=\mathbb{P}^1$, together with its toric boundary $\Delta$. Let $A$ be a  section of $X\to Z$ which does not contain any non-klt centre of $(X,\Delta)$, and let $G$ be a general fibre. Put $B=\Delta+ G$. Then $(X,B),A\to Z$ satisfies the assumptions of the theorem with $d=2, \Phi=\{1\}, u=1,v=1$. But it is well-known that such $X$ do not form a bounded family. 
\bigskip


\section{Preliminaries}

In this section we collect definitions and results that are used in later sections. We work over an algebraically closed field $k$ of characteristic zero. Varieties are assumed to be quasi-projective over $k$ unless stated otherwise.

\subsection{Contractions}
A \emph{contraction} means a projective morphism $f\colon X\to Y$ of varieties 
such that $f_*\mathcal{O}_X=\mathcal{O}_Y$ ($f$ is not necessarily birational). In this paper we will always work in the setting where $X$ is normal in which case so is $Y$. 

\subsection{Divisors}
Let $X$ be a normal variety, and let $M$ be an $\R$-divisor on $X$. For a prime divisor $D$ we denote its coefficient in $M$ by $\mu_DM$. 
    
We say that $M$ is \emph{b-Cartier} if it is $\R$-Cartier and if there is a birational contraction 
$\phi\colon W\to X$ from a normal variety such that $\phi^*M$  is Cartier.

For a birational map $X\bir X'$ (resp. $X\bir X''$, resp. $X\bir X'''$, and resp. $X\bir Y$) 
whose inverse does not contract divisors, and for 
an $\R$-divisor $M$ on $X$ we usually denote the pushforward of $M$ to $X'$ (resp. $X''$, resp. $X'''$, and resp. $Y$) 
by $M'$ (resp. $M''$, resp. $M'''$, and resp. $M_Y$).

\subsection{b-divisors}\label{ss-b-divisor}

Let $X$ be a normal variety. A \emph{b-divisor} $\mathbf{M}$ over $X$ is a collection of $\R$-divisors 
$M_Y$ on $Y$ for each birational contraction $Y\to X$ from a normal variety that are compatible 
with respect to pushforward, 
that is, if $Y'\to X$ is another birational contraction and $\psi\colon Y'\bir Y$ is a morphism, 
then $\psi_*M_{Y'}=M_Y$. 
 
A b-divisor $\mathbf{M}$ is \emph{b-$\R$-Cartier} if there is a birational contraction $Y\to X$ 
such that ${M}_Y$ is $\R$-Cartier and such that for any birational contraction 
$\psi\colon Y'\to Y/X$ we have ${M}_{Y'}=\psi^*M_Y$. 
In other words, a b-$\R$-Cartier b-divisor over $X$ is determined by the choice of a birational contraction  
$Y\to X$ and an $\R$-Cartier $\R$-divisor $M$ on $Y$. But this choice is not unique,  
that is, another birational contraction $Y'\to X$ and an $\R$-Cartier $\R$-divisor
$M'$ on $Y'$ defines the same b-$\R$-Cartier b-divisor if there is a common resolution $W\to Y$ and $W\to Y'$ 
on which the pullbacks of $M$ and $M'$ coincide.  

A b-$\R$-Cartier  b-divisor  represented by some $Y\to X$ and $M$ is \emph{b-Cartier} if  $M$ is 
b-Cartier, i.e. its pullback to some resolution is Cartier.

\subsection{Pairs}
A \emph{pair} $(X,B)$ consists of a normal quasi-projective variety $X$ and an $\R$-divisor 
$B\ge 0$ such that $K_X+B$ is $\R$-Cartier. 
If the coefficients of $B$ are at most $1$ we say $B$ is a 
\emph{boundary}.

Let $\phi\colon W\to X$ be a log resolution of a pair $(X,B)$. Let $K_W+B_W$ be the 
pulback of $K_X+B$. The \emph{log discrepancy} of a prime divisor $D$ on $W$ with respect to $(X,B)$ 
is $1-\mu_DB_W$ and it is denoted by $a(D,X,B)$.
We say $(X,B)$ is \emph{lc} (resp. \emph{klt}, resp. \emph{$\epsilon$-lc}) 
if $a(D,X,B)$ is $\ge 0$ (resp. $>0$, resp. $\ge \epsilon$) for every $D$. 
Note that if $(X,B)$ is an lc pair, then 
the coefficients of $B$ necessarily belong to $[0,1]$. Also if $(X,B)$ is $\epsilon$-lc, then 
automatically $\epsilon\le 1$ because $a(D,X,B)\leq 1$ for any divisor $D$ on $X$. 

Let $(X,B)$ be a pair. A \emph{non-klt place} of $(X,B)$ is a prime divisor $D$ on 
a birational model of $X$ such that $a(D,X,B)\le 0$. A \emph{non-klt centre} is the image on 
$X$ of a non-klt place. The \emph{non-klt locus} of $(X,B)$ is the union of all the non-klt centres.

\subsection{Minimal model program (MMP)}\label{ss-MMP} 
We will use standard results and concepts of the minimal model program (cf. \cite{KM98} and \cite{BCHM10}). 
Assume $(X,B)$ is a pair, $X\to Z$ is a projective morphism, 
 $H$ is an ample$/Z$ $\R$-divisor, and $K_X+B+H$ is nef$/Z$. Suppose $(X,B)$ is klt or 
that it is $\Q$-factorial dlt. We can run an MMP$/Z$ on $K_X+B$ with scaling of $H$. If 
$(X,B)$ is klt and if either $K_X+B$ or $B$ is big$/Z$, then the MMP terminates \cite{BCHM10}. If $(X,B)$
is $\Q$-factorial dlt, then in general we do not know whether the MMP terminates but 
we know that after finitely many steps of the MMP we reach a model $Y$ on which $K_Y+B_Y$, 
the pushforward of $K_X+B$, is a limit of movable$/Z$ $\R$-divisors: indeed, if the MMP terminates, then 
the claim is obvious; otherwise the MMP produces an infinite sequence $X_i\bir X_{i+1}$ 
of flips and a decreasing sequence $\alpha_i$ of numbers in $(0,1]$ such that 
$K_{X_i}+B_i+\alpha_iH_i$ is nef$/Z$; by \cite{BCHM10} and \cite[Theorem 1.9]{B12}, $\lim\alpha_i=0$; 
in particular, if $Y:=X_1$, then $K_Y+B_Y$ is the limit of the movable$/Z$ $\R$-divisors 
$K_Y+B_Y+\alpha_i H_Y$.

\subsection{Generalised pairs}\label{ss-gpp}
For the basic theory of generalised pairs see \cite[Section 4]{BZ16} and for a survey see 
\cite{B21b}. A \emph{generalised pair} consists of 
\begin{itemize}
\item a normal variety $X$ equipped with a projective
morphism $X\to Z$, 

\item an $\R$-divisor $B\ge 0$ on $X$, and 

\item a b-$\R$-Cartier  b-divisor over $X$ represented 
by some birational contraction $X' \overset{\phi}\to X$ and an $\R$-Cartier divisor
$M'$ on $X'$ such that $M'$ is nef$/Z$ and $K_{X}+B+M$ is $\R$-Cartier,
where $M := \phi_*M'$.
\end{itemize}

We will often refer to such a generalised pair by saying that $(X,B+M)$ is a generalised pair with data $X'\to X$ and $M'$. When $Z$ is a point we omit it and simply say that the pair is projective in which case we also say that $(X,B+M)$ is a generalised pair with nef part $M'$.

 Since a b-$\R$-Cartier b-divisor is defined birationally, 
in practice we will often replace $X'$ with a resolution and replace $M'$ with its pullback.

Now we define generalised singularities.
Replacing $X'$ we can assume $\phi$ is a log resolution of $(X,B)$. We can write 
$$
K_{X'}+B'+M'=\phi^*(K_{X}+B+M)
$$
for some uniquely determined $B'$. For a prime divisor $D$ on $X'$ the \emph{generalised log discrepancy} $a(D,X,B+M)$ is defined to be $1-\mu_DB'$. 

We say $(X,B+M)$ is 
\emph{generalised lc} (resp. \emph{generalised klt}, resp. \emph{generalised $\epsilon$-lc}) 
if for each $D$ the generalised log discrepancy $a(D,X,B+M)$ is $\ge 0$ (resp. $>0$, resp. $\ge \epsilon$).
A \emph{generalised non-klt place} of $(X,B+M)$ is a prime divisor 
$D$ on a birational model of $X$ with 
$$
a(D,X,B+M)\le 0,
$$ 
and a \emph{generalised non-klt centre} of $(X,B+M)$ is the image of a generalised non-klt place.
The \emph{generalised non-klt locus}  of the generalised pair is the union of all 
the generalised non-klt centres.  

Given a generalised lc pair $(X,B+M)$, we can run the MMP on 
$K_X+B+M$ over $Z$ with scaling of an ample divisor if $(X,C)$ is klt for some boundary $C$, 
for example, this is the case when $(X,B+M)$ is $\Q$-factorial generalised dlt 
or when it is generalised klt. The termination of this MMP has not been established in full generality, however it is known that the MMP terminates if $(X,B+M)$ is generalised klt and  
$K_X+B+M$ or $B+M$ is big over $Z$. 
For a more detailed discussion of the MMP on generalised pairs, see \cite[Lemma 4.4]{BZ16}.

\subsection{Families of generalised pairs}\label{d-fam-glc-model}
Let $d\in \mathbb{N}$, $\Phi\subset \mathbb{Q}^{\ge 0}$, and $v\in \Q^{>0}$. 
Let $\mathcal{G}_{glc}(d,\Phi)$ be the set of projective 
generalised pairs $(X,B+M)$ with nef part $M'$ such that 
\begin{itemize}
\item $(X,B+M)$ is generalised lc of dimension $d$, 
\item the coefficients of $B$ are in $\Phi$, and
\item $M'=\sum \mu_i M_i'$ where $\mu_i\in \Phi$ and $M_i'$ are nef Cartier.
\end{itemize}
Let 
$\mathcal{F}_{glc}(d,\Phi,v)$ be the set of those $(X,B+M)\in \mathcal{G}_{glc}(d,\Phi)$
such that
\begin{itemize}
\item $K_X+B+M$ is ample with volume $\vol(K_X+B+M)=v$.
\end{itemize}
We define $\mathcal{F}_{gklt}(d,\Phi,v)$ similarly by replacing the generalised lc condition with generalised klt, 
and further define 
$\mathcal{F}_{gklt}(d,\Phi,<v)$ similarly by replacing the condition 
$$
\vol(K_X+B+M)=v
$$ 
with 
$$
\vol(K_X+B+M)<v.
$$

\subsection{Bounded families}\label{ss-bnd-couples}
A \emph{couple} $(X,D)$ consists of a normal projective variety $X$ and a  divisor 
$D$ on $X$ whose non-zero coefficients are all equal to $1$, i.e. $D$ is a reduced divisor. 
The reason we call $(X,D)$ a couple rather than a pair is that we are concerned with 
$D$ rather than $K_X+D$ and we do not want to assume $K_X+D$ to be $\Q$-Cartier 
or with nice singularities. Two couples $(X,D)$ and $(X',D')$ are isomorphic if 
there is an isomorphism $X\to X'$ mapping $D$ onto $D'$.

We say that a set $\mathcal{P}$ of couples of dimension $\le d$ is \emph{bounded} 
if there is $r\in \mathbb N$ such that for each $(X,D)\in \mathcal{P}$ we can find a very 
ample divisor $A$ on $X$ so that $A^{\dim X}\le r$ and $D\cdot A^{\dim X-1}\le r$. This is  
equivalent to saying that there exist 
finitely many projective morphisms $V^i\to T^i$ of varieties and reduced divisors $C^i$ on $V^i$ such that for each $(X,D)\in \mathcal{P}$ there exist an $i$, a closed point $t\in T^i$, and an 
isomorphism $\phi\colon V^i_t\bir X$ such that $(V^i_t,C^i_t)$ is a couple and 
$\phi_*C_t^i\ge D$.

A set $\mathcal{Q}$ of projective lc pairs $(X,B)$ is said to be bounded if the $(X,\Supp B)$ 
form a bounded family of couples.

Assume $\Phi\subset \mathbb{Q}^{\ge 0}$ such that $0$ is not an accumulation point of $\Phi$, e.g. when $\Phi$ is DCC.  
We say that a set $\mathcal{E}\subset \mathcal{G}_{glc}(d,\Phi)$ forms a bounded family if 
there is $r\in \mathbb N$ such that for each $(X,B+M)\in \mathcal{E}$ there is a very ample 
divisor $A$ on $X$ with 
$$
A^d\le r \ \ \mbox{and} \ \  (K_X+B+M)\cdot A^{d-1}\le r.
$$ 
In particular this implies that the $(X,\Supp B)$ form a bounded family of couples. However, since $M$ is not necessarily effective, we cannot control $\Supp M$. In practice we can only bound $\Supp M$ up to $\Q$-linear equivalence. 

Some of the main results on boundedness of varieties can be found in \cite{HMX18}, \cite{B21a}, \cite{B20} where general type, Fano, and Calabi-Yau pairs are treated, respectively (see also \cite{HMX13},\cite{HMX14}), \cite{B19}).

\subsection{Volume of divisors} For a $\Q$-divisor $D$ on a normal projective variety with Kodaira dimension $\kappa(D)\ge 0$, we define the Iitaka volume to be 
$$
{\rm Ivol}(D)=\limsup_{m\in \mathbb{N}}\frac{\kappa(D)! \ h^0(X,mD)}{m^{\kappa(D)}}.
$$
When $D$ is big, this is also called volume of $D$ and denoted $\vol(D)$. When $D$ is semi-ample, it defines a contraction $f\colon X\to Z$ and $D\sim_\Q f^*H$ for some ample $\Q$-divisor $H$; in this case, $\Ivol(D)=H^{\dim Z}$.

\subsection{Semi-ample families}
A semi-ample family $(X,B)\to Z$ consists of an lc pair $(X,B)$ and a contraction $X\to Z$ such that $K_X+B$ is semi-ample$/Z$.

\begin{lem}\label{l-s-ample-family}
Let $(X,B)$ be a pair and $X\to Z$ be a contraction. Assume that there is a dense set of 
closed points $z_i\in Z$ so that $K_{F_i}+B_{F_i}:=(K_X+B)|_{F_i}$ is semi-ample and $(F_i,B_{F_i})$ is lc for each $i$, where $F_i$ is the fibre over $z_i$. 
Then $K_X+B$ is semi-ample over some non-empty open subset of $Z$.  
\end{lem}
\begin{proof}
Shrinking $Z$ we can assume that $(X,B)$ is lc because a dense set of its log fibres are lc. 
Moreover, replacing $(X,B)$ with a dlt model, we can assume $(X,B)$ is dlt. 
By the dlt condition, there is a log resolution $\phi\colon W\to X$ so that $a(D,X,B)>0$ 
for every prime exceptional divisor $D$ of $\phi$. 
Let $B_W$ be the sum of the reduced exceptional divisor of $\phi$ and the birational transform of $B$.
Shrinking $Z$, we can assume that $(W,B_W)$ is relatively log smooth over $Z$. 

Let $F_i,G_i$ be the fibres of $X\to Z$ and $W\to Z$ over $z_i$, respectively. 
Let $K_{F_i}+B_{F_i}:=(K_X+B)|_{F_i}$ and $K_{G_i}+B_{G_i}:=(K_W+B_W)|_{G_i}$. 
Shrinking $Z$, we can assume that $(F_i,B_{F_i})$ is a good minimal model of $(G_i,B_{G_i})$ 
for every $i$. Then by \cite[Theorem 1.2]{HMX18}, $(W,B_W)$ has a good minimal model 
$(Y,B_Y)$ over $Z$. Thus $(Y,B_Y)$ is also a good minimal model of $(X,B)$, and running an MMP 
on $K_X+B$ over $Z$ with scaling of an ample divisor ends with a good minimal model \cite[Theorem 1.9]{B12}. Replace $(Y,B_Y)$ with this minimal model. Since the $K_{F_i}+B_{F_i}$ are nef,
we can assume that the extremal rays contracted in the course of the MMP do not intersect any of the $F_i$, 
i.e. the curves generating such rays do not intersect $F_i$. Therefore, these extremal rays do not intersect 
the generic fibre of $X\to Z$. So shrinking $Z$ we can assume that $X\bir Y$ is an isomoprhism, hence 
we can assume $K_X+B$ is already semi-ample over $Z$.
 
\end{proof}

\begin{lem}\label{l-universal-family}
Let $d,v\in \mathbb N$ and $\Phi\subset \Q^{\ge 0}$ be a finite set.  
Assume that $f_i:(X_i,B_i)\to Z_i$ is a sequence of semi-ample families and $H_i$ are divisors on $X_i$ such that 
\begin{itemize}
\item $d=\dim X_i>\dim Z_i$,
\item the horizontal$/Z_i$ coefficients of $B_i$ are in $\Phi$, 
\item $H_i$ is very-ample over $Z_i$ with 
$$
(H_i|_{F_i})^{\dim F_i}\le v \ \ \mbox{and} \ \ 
(H_i|_{F_i})^{\dim F_i-1}\cdot B_i|_{F_i}\le v
$$ 
for the general fibres $F_i$ of $X_i\to Z_i$.
\end{itemize}
Then, perhaps after replacing the sequence with an infinite subsequence, there exists a semi-ample family 
$(\mathcal{V},\mathcal{B})\to V$ such that for each $i$ we can find a rational map $Z_i\bir V$ so that over the generic point $\eta_{Z_i}$ the family $(X_i,B_i)\to Z_i$ is the pullback of 
$(\mathcal{V},\mathcal{B})\to V$ via $Z_i\bir V$. Moreover, the images of $\eta_{Z_i}$ in $V$ form a dense set.
\end{lem}
\begin{proof}
Note that pullback here means the following where we shrink $Z_i$ if necessary. 
First $X_i=Z_i\times_V\mathcal{V}$. 
Let $\mathcal{V}^0$ be the largest open subset of $\mathcal{V}$ on which 
$(\mathcal{V},\mathcal{B})\to V$ is relatively log smooth, and let $\mathcal{B}^0=\mathcal{B}|_{\mathcal{V}^0}$. 
Then taking the closure of $B_i^0=Z_i\times_V\mathcal{B}^0$ inside $X_i$ gives $B_i$.

Our assumptions imply that the log general fibres $(F_i,B_{F_i})$ of the $f_i$ all belong to a bounded family of pairs. 
In particular, replacing the sequence $\Xi:=\{(X_i,B_i)\to Z_i\}$ with an infinite subsequence 
and perhaps shrinking $Z_i$, discarding vertical components of $B_i$ and rearranging the indices, 
we can assume $B_i=\sum_{j=1}^l b_jB_{i,j}$ where $B_{i,j}$ are the irreducible components of 
$B_i$ and the $b_j$ are independent of $i$. Moreover, we can assume that $r:=\dim F_i$ is independent of $i$. 
Let $\alpha=(b_1,\dots,b_l)$ where we allow the possibility that $l=0$ and $\alpha=\emptyset$. 
Then further shrinking $Z_i$ we see that $(X_i,B_i)\to Z_i$ is a $(r,\alpha)$-marked locally stable family according to \cite[7.1(3)]{B20}.

On the other hand, perhaps after replacing $H_i$ with a bounded multiple and shrinking $Z_i$ 
we can assume that $H_i$ defines an embedding $g_i\colon X_i\to \mathbb{P}^n_{Z_i}$, for some fixed $n$, 
so that $R^kf_{i_*}\OO_{X_i}(H_i)\simeq R^k\pi_{i_*}\OO_{\mathbb{P}^n_{Z_i}}(1)$ for all $k\ge 0$ where $\pi_i: \mathbb{P}^n_{Z_i}\to Z_i$ is the projection and $f_i=\pi _i\circ g_i$. 
Then $(X_i,B_i)\to Z_i$ is a strongly embedded $(r,\alpha,\mathbb{P}^n)$-marked 
locally stable family according to \cite[7.1(5)]{B20}.

The moduli functor of strongly embedded $(r,\alpha,\mathbb{P}^n)$-marked 
locally stable families over reduced schemes is represented by a reduced separated scheme ${E}$ equipped with a universal family $(\mathcal{E},\mathcal{D})\to E$ \cite{K21} (also see \cite[Theorem 7.2]{B20}). In particular, each $(X_i,B_i)\to Z_i$ is the pullback of $(\mathcal{E},\mathcal{D})\to E$ with respect to some morphism $Z_i\to E$ (this kind of pullback is explained in \cite[7.1(2)]{B20}).  Our assumptions 
$$
(H_i|_{F_i})^{\dim F_i}\le v ~~\mbox{and}~~ (H_i|_{F_i})^{\dim F_i-1}\cdot B_i|_{F_i}\le v
$$ 
ensure that all the $Z_i$ map into a fixed finite type subscheme of $E$.
Therefore, there exists a strongly embedded $(r,\alpha,\mathbb{P}^n)$-marked 
locally stable family $(\mathcal{V},\mathcal{B})\to V$ where $\mathcal{V},V$ are of finite type, so that each family $(X_i,B_i)\to Z_i$ is 
the pullback of $(\mathcal{V},\mathcal{B})\to V$ via some morphism $\lambda_i\colon Z_i\to V$.  
Moreover, replacing the sequence $\Xi$ with an infinite subsequence and replacing $V$ with the closure
of the set of $\lambda_i(\eta_{Z_i})$, we can assume that $V$ is irreducible and smooth and 
that $\{\lambda_i(\eta_{Z_i})\}$ is dense in $V$. In addition, we can assume that $\mathcal{V}$ is irreducible and that $(\mathcal{V},\mathcal{B})$ is an lc pair. Also we can assume $\mathcal{V}\to V$ is a contraction. Now $(X_i,B_i)\to Z_i$ is 
the pullback of $(\mathcal{V},\mathcal{B})\to V$ via $\lambda_i$ with the meaning explained in the beginning of this proof. In particular, the pullback of $K_{\mathcal{V}}+\mathcal{B}$ to $X_i$ is $K_{X_i}+B_i$.

It is enough to show that perhaps after shrinking $V$, $(\mathcal{V},\mathcal{B})\to V$ is a semi-ample family, 
that is, that $K_{\mathcal{V}}+\mathcal{B}$ is semi-ample over $V$. Since the 
$\lambda_i(\eta_{Z_i})$ form a dense set in $V$ and since the general log fibres $(F_i,B_{F_i})$ of 
$(X_i,B_i)\to Z_i$ are among the general log fibres of $(\mathcal{V},\mathcal{B})\to V$, 
we see that there is a dense set of closed points $v_i\in V$ so that 
$K_{\mathcal{V}}+\mathcal{B}$ restricted to the fibres over the $v_i$ is semi-ample. Now apply 
Lemma \ref{l-s-ample-family}.
  
\end{proof}

\section{Failure of DCC of Iitaka volumes}

In this section we will study how Iitaka volumes behave in families. 
Consider the setup of \ref{setup} so that $(X,{\rm Supp }B +{\rm Supp}M)$ is log smooth over $T$ and suppose that for a dense set of closed points $T'\subset T$, $K_{X_t}+B_t+M_t$ is nef and big for 
all $t\in T'$. Letting $\eta$ be the generic point of $T$, we have that $K_{X_\eta }+B_\eta +M_\eta $ is nef. Moreover,  for $t\in T'$ the volume  
\[\vol (K_{X_t}+B_t+M_t)=\vol (K_{X_\eta }+B_\eta +M_\eta)=(K_{X_\eta }+B_\eta +M_\eta)^{\dim X_\eta }\] 
is constant (actually $\vol(K_{X_t}+B_t+M_t)$ is constant for $t\in T'$ even without assuming $K_{X_t}+B_t+M_t$ is nef, by \cite[Theorem 1.12]{Fil18} which relies on \cite[Theorem 1.8]{HMX13}). 

In much greater generality, if $d\in \N$ and $\Phi\subset \Q^{\ge 0}$ is DCC, 
then 
$$
\{\vol(K_X+B+M) \mid (X,B+M)\in \mathcal{G}_{glc}(d,\Phi)\}
$$
is a DCC set, by \cite[Theorem 1.3]{B21}.

If we remove the bigness assumption, then it is natural to consider behaviour of the Iitaka volumes.
In the case of usual pairs, that is, when $M=0$, assuming that $K_{X_t}+B_t$ is semiample for $t\in T'$, it follows by \cite[Theorem 1.2]{HMX18} that $(X,B)$ has a good minimal model over $T$. It is then easy to see that the Iitaka volumes $\Ivol(K_{X_t}+B_t)$, for $t\in T'$, belong to a DCC set, by stratifying the base $T$ appropriately.  

However, the following example shows that when $M\ne 0$, then the DCC property and the existence of 
a lower bound for the Iitaka volumes $\Ivol(K_{X_t}+B_t+M_t)$ can fail (even though we assume that each $K_{X_t}+B_t+M_t$ is semiample and generalised klt for all $t\in T'$).

\begin{exa}\label{exa-Iitaka-volume}
\emph{Let $m_i$ be a strictly increasing sequence of natural numbers. 
We will see that there exists a bounded family of surfaces $X_i$ so that 
$|-m_iK_{X_i}|$ is base point free defining a contraction $f_i\colon X_i\to Z_i$ 
where $f_i$ has a multiple fibre with multiplicity $m_i$ (i.e. the fibre is  
$m_iG_i$ where $G_i$ is reduced and irreducible) and that letting $M_i=-2K_{X_i}$, 
we have 
$$
{\rm Ivol}(K_{X_i}+M_i)=\deg (K_{Z_i}+B_{Z_i}+M_{Z_i})=\frac 1 {m_i}
$$ 
where the expression $K_{Z_i}+B_{Z_i}+M_{Z_i}$ is given by the generalised canonical 
bundle formula applied to $(X_i,M_i)\to Z_i$ as in \cite{Fil20}}. 
 
\emph{To define $X_i$ we follow \cite[Main Theorem 2.1]{Fu90}. Pick a smooth elliptic curve $E\subset \mathbb P ^2$ and nine points $p_1,\ldots , p_9\in E$ such that $\sum p_i$ has order $m_i$ and in particular $m_i\sum p_i=O$ where $O$ is the identity of $E$ (here $\sum p_i$ means sum with respect to the group structure on $E$, not the sum as a divisor). We may also fix a curve $C$ of degree $3m_i$ with simple singularities of order $m_i$ at each $p_i$ and no other singularities.
If we blow up $\mathbb P ^2$ we obtain a surface $X_i$ with a morphism $f_i:X_i\to Z_i\cong \mathbb P ^1$ defined by a pencil contained in the base point free linear series $|-m_iK_{X_i}|$. This pencil is spanned by $m_iE'$ and $C'$ the strict transforms of $m_iE$ and $C$, thus there is a unique fiber of multiplicity $m_i$ and all other fibers are reduced. We also have $-m_iK_{X_i}=f_i^*\mathcal O _{\mathbb P ^1}(1)$. 
Letting $A_i=2E'+F'$ where $F'$ is the exceptional divisor of $X_i\to \mathbb{P}^2$, then $A_i$ is very ample (this follows easily by Reider's Criterion). We have ${\rm vol}(A_i)=A_i^2=27$.
In particular, such $X_i$ form a bounded family.}

\emph{Let $M_i=-2K_{X_i}$. Consider the generalised canonical bundle formula 
$$
-K_{X_i}=K_{X_i}+M_i\sim_\Q f_i^*(K_{Z_i}+B_{Z_i}+M_{Z_i}).
$$
By construction, 
$$
\deg (K_{Z_i}+B_{Z_i}+M_{Z_i})=\frac 1 {m_i},
$$
$B_{Z_i}=(1-\frac 1 {m_i})Q_i$ where $Q_i$ is the image of $E'$, and $\deg M_{Z_i}=1+\frac 2 {m_i}$.} 

\emph{Notice that the Iitaka volume  
$$
{\rm Ivol}(K_{X_i}+M_i)=\deg (K_{Z_i}+B_{Z_i}+M_{Z_i})
$$ 
is not bounded from below. 
It is easy to see that the set of such $(X_i,M_i)$ belongs to a family as in \ref{setup} where $(X_i,M_i)$ is the fibre over some $t_i\in T$ and the $t_i$ are dense in $T$ (see \S 4 for details).}

\emph{Note that $\vol(A_i|_{F_i})=A_i\cdot F_i=9m_i$ is not bounded from above where $F_i$ is a 
general fibre of $f_i$.  
So, not surprisingly, the Cartier index of $M_{Z_i}$ is not bounded. 
See \cite[Theorem 1.7 and Lemma 7.4]{B21}, for situations for usual pairs where from boundedness of volume on general fibres we can derive boundedness of Cartier index of the moduli divisor and subsequently DCC property of the Iitaka volume.}
\end{exa}

On the other hand, one wonders if we can have a better behaviour of the Iitaka volume 
numerically:

\begin{ques}
Adopting the notation of \ref{setup}, assume that $(X_t,B_t+M_t)$ has a good minimal model for 
all $t\in T'$, in particular 
$$
\kappa (K_{X_t}+B_t+M_t)=\nu(K_{X_t}+B_t+M_t).
$$ 
Is it true that 
$$
\kappa (K_{X_\eta }+B_\eta+M^*_\eta)=\nu(K_{X_\eta}+B_\eta+M_\eta)
$$ 
for some $M^*_\eta \equiv M_\eta$ where $\eta\in T$ is the generic point?

\end{ques}

We already pointed out above that the answer to the question is negative if we take 
$M^*_\eta=M_\eta$.

\section{Failure to the existence of a relative nef fibration}

A useful fact about nef divisors on projective varieties is the existence of the so-called 
nef fibration (or nef reduction) map \cite{Nef} which roughly speaking contracts the curves with trivial 
intersection with the nef divisor, at least, generically. More precisely, let $L$ be a 
nef divisor on a normal projective variety $X$. Then by \cite[Theorem 2.1]{Nef}, there is a rational map $X\bir Y$ 
which is a morphism over some non-empty open subset $U\subset Y$, such that
\begin{enumerate}
    \item $L|_F\equiv 0$ for the 
fibres $F$ over $U$, and
\item any curve $C$ passing through a very general point of $X$ such that $L\cdot C=0$ is 
in fact contracted to a point by $X\bir Y$.
\end{enumerate}     
Note that if $\dim X=2$, then the second condition above applies for curves through a general point of $X$ (if we work over $\C$, it is in fact a well known consequence of the Hodge Index Theorem that on a surface, the set of curves $C$ such that $C\cdot L=0$ is either finite or uncountable). However, when $\dim X\geq 3$, \cite[Theorem 1]{LO16} shows that it is indeed necessary to consider curves $C$ passing through very general points of $X$.
Here we discuss an example in the relative setting also exhibiting a similar behaviour.  

Fix $E\subset \mathbb P^2$ a smooth elliptic curve and distinct points $p_1,\ldots , p_8$ on $E$ in general position, i.e. no 3 of then on a line and no 6 of them on a conic. Consider the projection
$$
\mathbb P^2\times T\to T\cong E\setminus \{p_1,\ldots , p_8\}.
$$  
Let 
$
f: X \to T
$ 
be the family of surfaces given by blowing up the constant sections 
$p_1\times T,\ldots , p_8\times T$ and the diagonal section 
$$
\Gamma =\{(p_t,t)\subset \mathbb{P}^2\times T \mid t\in T\}
$$ 
in $\mathbb P^2\times T$ where $p_t\in E$ is the point corresponding to $t\in T$ under the given isomorphism $T\cong E\setminus \{p_1,\ldots , p_8\} $. Thus for any $t\in T$ we have that
$X_t$ is the blow up of $\mathbb P^2$ along $p_1,\ldots , p_8, p_t$. 
We let $\mathcal E $ be the strict transform of $E\times T$ and $\mathcal F$ the exceptional divisor for $\nu:X\to \mathbb P^2\times T$, then 
$2\mathcal E+\mathcal F$ is $f$-very ample of degree $(2\mathcal E_t+\mathcal F_t)^2=27$.
Let 
\[T'=\{t\in T \mid p_1+\ldots +p_8+p_t \ {\rm is\ torsion\ on\ }E\},\qquad T^*=T\setminus T'.
\]
Here $p_1+\ldots +p_8+p_t$ is sum in $E$ according to the group law on $E$.

\begin{claim} $-K_{X} $ is $f$-nef. More precisely
\begin{enumerate}
\item If $t\in T'$, then $-K_{X_t}$ is semiample and $|-mK_{X_t}|$ defines a morphism to $\mathbb P^1$ with one multiple fiber corresponding to $m\mathcal E_t$ where $m$ is the order of the torsion point 
$ p_1+\ldots +p_8+p_t$ in $E$.
\item If $t\in T^*$, then $-K_{X_t}$ is nef. Moreover, if $t\in T^*$ and $p_1,...,p_8, p_t$ are in very general position on $E$ then $-K_{X_t}\cdot C>0$ for all  curves $\mathcal E_t\ne C\subset X_t$.
\end{enumerate}  
\end{claim}
\begin{proof}  (1) has been verified above, and since 
$\mathcal E_t\sim -K_{X_t}$ and $\mathcal E_t^2=0$, it follows that $-K_{X_t}$ is nef.  
So suppose that $ p_1+\ldots +p_8+p_t$  are in very general position on $E$. If $-K_{X_t}\cdot C=0$ for some curve $\mathcal E_t\ne C\subset X_t$, then the pushforward $D$ of $C$ on $\mathbb{P}^2$ is a curve  such that $D\cdot E=\sum _{i=1}^9m_ip_i$. Since $g(E)>0$, this is well known to be impossible (see eg. \cite[Lemma 5]{Nag59}).

\end{proof}

Let $M=-2K_{X}$, then 
$$
L:=K_{X}+M=-K_{X}
$$ 
is $f$-nef. 
By what we have seen above, the set of irreducible curves  $C\subset X_t$ such that $L\cdot C=0$ consists of one curve for each $t\in T^*$ such that $p_1,...,p_8, p_t$ are in very general position on $E$ and consists of the fibers of the
induced fibrations $X_t\to \mathbb P ^1$ for each $t\in T'$.
Thus a relative nef fibration for $L$ does not exist because $T'$ is dense in $T$.

\section{Generalised lc models}

In this section we study boundedness of generalised lc models, that is, projective generalised lc 
pairs $(X,B+M)$ with ample $K_X+B+M$.

\subsection{Boundedness}
Fix $d\in \mathbb{N}$, $\Phi\subset \Q^{\ge 0}$ a DCC set, and $v\in \Q^{>0}$. 
Let $\mathcal{F}_{glc}(d,\Phi,v)$ be the set of projective generalised pairs $(X,B+M)$ 
with data $X'\overset{\phi}\to X$ and $M'$ where 
\begin{itemize}
\item $(X,B+M)$ is generalised lc of dimension $d$, 
\item the coefficients of $B$ are in $\Phi$, 
\item $M'=\sum \mu_iM_i'$ where $M_i'$ are nef Cartier and $\mu_i\in \Phi$,  
\item $K_X+B+M$ is ample, and 
\item  
$
\vol(K_X+B+M)=v.
$ 
\end{itemize}

Let 
$$
\mathcal{F}_{gklt}(d,\Phi,v)\subset \mathcal{F}_{glc}(d,\Phi,v)
$$ 
consist of the subset of pairs that are generalised klt. 

For usual pairs, that is, when $M'=0$, it is well-known that the corresponding set of pairs 
in $\mathcal{F}_{glc}(d,\Phi,v)$ forms a bounded family \cite{HMX18}. It is then natural to ask whether 
$\mathcal{F}_{glc}(d,\Phi,v)$ also forms a bounded family, that is, whether 
for each $(X,B+M)$ we can find a very ample divisor $A$ so that $A^d$ and $A^{d-1}\cdot(K_X+B+M)$ 
are bounded from above depending on $d,\Phi,v$. 

Boundedness in the generalised klt case is known, that is, 
$\mathcal{F}_{gklt}(d,\Phi,v)$ is bounded, see \cite[Theorem 1.4]{B21}.
Surprisingly, we will show that $\mathcal{F}_{glc}(d,\Phi,v)$ is not bounded in general by giving a counter-example 
in dimension 3. 

In \cite{Fil18} it is claimed that boundedness holds in dimension $2$, that is, 
$\mathcal{F}_{glc}(2,\Phi,v)$ is bounded but the proof overlooks some subtle issues so boundedness for generalised lc pairs in dimesion 2
is an open question. 

\begin{ques} 
Is $\mathcal{F}_{glc}(2,\Phi,v)$ bounded?
\end{ques}

We will show that the Cartier index of $K_X+B+M$ for 
$$
(X,B+M)\in \mathcal{F}_{glc}(2,\Phi,v)
$$ 
is not always bounded even when $(X,B)$ is fixed but $M$ varies. 

Although boundedness of $\mathcal{F}_{glc}(d,\Phi,v)$ fails in general but we still hope that some subsets of interest are bounded. 

\begin{ques}
Do the generalised pairs in $\mathcal{F}_{glc}(d,\Phi,v)$ given by the canonical bundle formula form a bounded family? More precisely, consider those $$
(X,B+M)\in \mathcal{F}_{glc}(d,\Phi,v)
$$
for which there is an lc pair $(V,\Delta)$ of fixed dimension with a contraction $f\colon V\to X$ so that the coefficients of $B$ are in a fixed DCC set and that we have a canonical bundle formula 
$$
K_V+\Delta\sim_\Q f^*(K_X+B+M).
$$
Do such $(X,B+M)$ form a bounded family?
\end{ques}

A possible answer is given by Theorem \ref{t-bnd-base-adjunction}.

In a different direction we ask:
\begin{ques}\label{q-bnd-gldt-model} 
Can we choose a generalised dlt model $(Y,B_Y+M_Y)$ for each 
$$
(X,B+M)\in \mathcal{F}_{glc}(d,\Phi,v)
$$ 
so that the $(Y,B_Y+M_Y)$ form a bounded family? 
\end{ques}

In dimension two the answer to this question is affirmative by the birational boundedness of $\mathcal{F}_{glc}(2,\Phi,v)$ discussed below. In dimension 3 again the answer is affirmative but for more 
subtle reasons. 

Another manifestation of boundedness is the fact that the set of log discrepancies
$$
\{a(D,X,B+M)\le 1\mid (X,B+M)\in \mathcal{F}_{glc}(d,\Phi,v), ~~~\mbox{$D$ prime divisor over $X$}\}
$$
is finite, see \cite[Theorem 1.5]{B21}. 

On the other hand, we know that the family $\mathcal{F}_{glc}(d,\Phi,v)$ is log birationally bounded, 
that is, there exists 
a bounded family of couples $\mathcal{P}$ so that 
for each 
$$
(X,B+M)\in \mathcal{F}_{glc}(d,\Phi,v)
$$
there exist a log smooth couple $(\overline{X},\overline{\Sigma})\in \mathcal{P}$ and a birational map 
$\overline{X}\bir X$ such that 
\begin{itemize}
\item $\overline{\Sigma}$ contains the exceptional divisors of $\overline{X}\bir X$ and the support of the birational 
transform of $B$, and 

\item every $M_i'$ with $\mu_i>0$ descends to a nef divisor $\overline{M}_i$ on $\overline{X}$.
\end{itemize}
See Theorem 1.2 of \cite{B21}.
Moreover, letting $\overline{B}$ be the birational transform of $B$ plus the reduced 
exceptional divisor of $\overline{X}\bir X$ and setting 
$\overline{M}=\sum \mu_i\overline{M}_i$, we can assume that 
$$
\vol(K_{\overline{X}}+\overline{B}+\overline{M})=v.
$$
See  \cite[Proposition 5.2]{B21}.

Using this birational boundedness, we can often reformulate our questions 
in terms of a fixed family. For example, a positive answer to the 
following question gives a positive answer to Question \ref{q-bnd-gldt-model}.

\begin{ques}
Assume $(X,B)$ is a pair where $B$ is a $\Q$-boundary and $M$ is a $\Q$-divisor such that 
\begin{itemize}
\item $(X,\Supp B\cup \Supp M)$ is relatively log smooth and projective over a variety $T$,
\item for a dense set of closed points $t\in T$, $M_t$ is nef, 
$K_{X_t}+B_t+M_t$ is big and $(X_t,B_t+M_t)$ has a generalised lc model 
(i.e. $R(K_{X_t}+B_t+M_t)$ is f.g.) where 
$X_t$ is the fibre over $t$, $B_t=B|_{X_t}$ and $M_t=M|_{X_t}$. 
\end{itemize} 
Perhaps after shrinking $T$, can we run an MMP on $K_X+B+M$ over $T$ so that it terminates?
\end{ques}

The failure of boundedness of $\mathcal{F}_{glc}(d,\Phi,v)$ shows that in general 
$(X_\eta,B_\eta+M_\eta)$ does not have a generalised lc model, that is, the 
generalised lc ring $R(K_{X_\eta}+B_\eta+M_\eta)$ fails to be finitely generated.
But this does not contradict the existence of the MMP as in the question above.

\subsection{Failure of boundedness of Cartier index for $\mathcal{F}_{glc}(2,\Phi,v)$}
We present an example to show that for 
$$
(X,B+M)\in \mathcal{F}_{glc}(2,\Phi,v)
$$ 
in general the Cartier 
index of $K_X+B+M$ is not bounded. Indeed, let $E$ be an elliptic curve over $\C$, 
$X$ be the projective cone over $E$, $Y\rightarrow X$ be the blow up of the vertex, and $Y\to E$ be the associated $\mathbb{P}^1$-bundle. Then $X$ has lc singularities: it is smooth other than at the 
vertex where it has an lc but not klt singularity. 
Pick a reduced divisor $B$ so that $K_X+B$ is ample and $(X,B)$ is lc, 
and let $K_Y+B_Y$ be its pullback to $Y$. Let $M_Y$ be the pullback of a torsion Cartier divisor on $E$ 
via $Y\to E$, and let $M$ be its pushforward to $X$. 
Then 
$$
(X,B+M)\in \mathcal{F}_{glc}(2,\Phi,v)
$$ 
where $\Phi=\{1\}$ and 
$$
v:=\vol(K_X+B+M)=\vol(K_X+B)
$$ 
is independent of $M$. But the Cartier index of 
$K_X+B+M$ depends on the Cartier index of $M$ which in turn depends on the torsion index of $M_Y$. 
In this example, $(Y,B_Y+M_Y)$ is the bounded generalised dlt model (as in Question \ref{q-bnd-gldt-model}), and $K_Y+B_Y+M_Y$ is semi-ample but it is not effectively 
semi-ample, i.e. no bounded multiple is base point free. 

This example also shows that in general a family $\mathcal{F}$ of generalised lc models $(X,B+M)$ can be bounded in the sense that there is a very ample divisor $A$ on $X$ with bounded $A^{\dim X}$ and $(K_X+B+M)\cdot A^{\dim X-1}$ but such that no bounded multiple of $K_X+B+M$ is very ample.

\subsection{Failure of boundedness of $\mathcal{F}_{glc}(3,\Phi,v)$}

We construct a family of generalised pairs 
$(U_i,B_{U_i}+N_{U_i})$ in $\mathcal{F}_{glc}(3,\Phi,v)$ 
for some $\Phi,v$ so that the $U_i$ do not form a bounded family. 
In fact, in the example $B_{U_i}=0$ and $N_{U_i}=-2K_{U_i}$, so $U_i$ is 
a Fano variety. 

Recall from Example \ref{exa-Iitaka-volume} that there exist a strictly increasing 
sequence $m_i$ of natural numbers and a bounded family of surfaces $X_i$ so that 
$|-m_iK_{X_i}|$ is base point free defining a contraction $f_i\colon X_i\to Z_i$ 
where $f_i$ has a multiple fibre with multiplicity $m_i$.   
Letting $M_i=-2K_{X_i}$, the canonical bundle formula gives 
$$
-K_{X_i}=K_{X_i}+M_i\sim_{\Qq} f_i^*(K_{Z_i}+B_{Z_i}+M_{Z_i}), 
$$
and 
$$
\deg (K_{Z_i}+B_{Z_i}+M_{Z_i})=\frac 1 {m_i}.
$$
Moreover, on $X_i$ we have a very ample divisor $A_i$ with $A_i^2=27$. 

Let 
$$
S_i={\rm Proj} (\mathcal{O}_{X_i}\oplus \mathcal{O}_{X_i}(A_i)),
$$
$\pi_i\colon S_i\to X_i$ be the projection, and $H_i$ be the tautological divisor. 
Then $H_i$ is big and base point free defining a contraction $S_i\to V_i$ 
which contracts exactly $X_i$ to a point. 

Consider 
$X_i$ as a subset of $S_i$ given by the section so that $H_i|_{X_i}\sim 0$. 
First note that 
$$
K_{S_i}+X_i+H_i
$$
is $\pi_i$-trivial, so 
$$
K_{S_i}+X_i+H_i\sim \pi_i^*K_{X_i},
$$
hence 
$-(K_{S_i}+X_i+H_i)$ 
is semi-ample. Therefore,
$$
-(K_{S_i}+X_i)=H_i-(K_{S_i}+X_i+H_i)
$$
is semi-ample and big as $H_i$ is big and base point free.

Now let $N_i=-2(K_{S_i}+X_i)$. Then 
$$
K_{S_i}+X_i+N_i=-(K_{S_i}+X_i)
$$ 
is semi-ample and big defining 
a contraction $h_i\colon S_i\to U_i$ which contracts $X_i$: actually the curves contracted are exactly 
the fibres of the elliptic fibration $X_i\to Z_i$  because if $C$ is any contracted curve, then 
$H_i\cdot C=0$ which means $C\subset X_i$, and $-K_{X_i}\cdot C=0$ 
(here $X_i$ is considered as embedded in $S_i$ as above). In particular, 
$$
K_{S_i}+X_i=h_i^*K_{U_i}~~~~\mbox{and}~~~~ N_i=h_i^*N_{U_i}
$$ 
where $N_{U_i}$ is the pushforward of $N_i$. So 
$Z_i\subset U_i$ is a generalised non-klt centre of $(U_i,N_{U_i})$ as well as of $(U_i,0)$. 
Note that $-K_{U_i}$ is ample, so $U_i$ is an lc Fano variety.

Since $(S_i,X_i)$ is bounded, replacing the sequence we can assume that 
the volume of $K_{S_i}+X_i+N_i$ is fixed, say it is $v$. Then the $(U_i,N_{U_i})$ are generalised 
lc models with fixed volume. More precisely, 
$$
(U_i,N_{U_i})\in \mathcal{F}_{glc}(3,\Phi,v)
$$
where $\Phi=\{0,1\}$.
However, by construction, 
$$
f_i^*(K_{U_i}|_{Z_i})\sim_\Q (h_i^*K_{U_i})|_{X_i}=(K_{S_i}+X_i)|_{X_i}=K_{X_i},
$$ 
so 
$$
-K_{U_i}\cdot Z_i={\rm Ivol}(-K_{X_i})=\deg (K_{Z_i}+B_{Z_i}+M_{Z_i})=\frac{1}{m_i}
$$ 
with the above notation, so the Cartier index of $K_{U_i}$ is not bounded. 

We argue that the $U_i$ cannot form a bounded family.
Suppose not, then we can assume that we have a projective family $U \to T$ 
containing all such $U_i$ as fibers $U _{t_i}$ over points $t_i\in T$. 
By the cone theorem, we can choose a very ample$/T$ divisor $A$ so that  
$K_{U_{t_i}}+A|_{U_{t_i}}$ is ample for every $i$. 

Let $\nu: U'\to U$ be a resolution. 
Shrinking $T$ and passing to a subsequence, we may assume that $T$ is smooth and $(U',E')$ is log smooth over $T$ where $E'$ is the exceptional divisor of $\nu$.
Changing $A$ linearly, we can assume that $A\ge 0$ and that $(U',E'+A')$ is log smooth over $T$ where 
$A'=\nu^*A$. 
Moreover, we may assume that 
$$
\nu _i=\nu |_{ U'_{t_i}}:U'_{t_i}\to  U_{t_i}\cong U_i
$$ 
is a resolution and 
$(U'_{t_i}, E'_{t_i}+A_{t_i}')$ is log smooth where $E'_{t_i}$ is the exceptional locus of $\nu _i$ 
and $A_{t_i}'=A'|_{U_{t_i}'}$. 

Passing to a subsequence we can assume that $\vol(K_{U'_{t_i}}+E'_{t_i}+A_{t_i}')$ 
is independent of $i$ \cite[Theorem 1.8]{HMX13}. Let $A_{t_i}$ be the pushforward of $A_{{t_i}}'$. 
Then $(U_{t_i}, A_{t_i})$ is lc because 
$$
\nu_i^*(K_{U_{t_i}}+A_{t_i})\le K_{U'_{t_i}}+E'_{t_i}+A_{t_i}'
$$   
where we use the facts that $U_{t_i}$ is lc and $A_{{t_i}}'=\nu_i^*A_{{t_i}}$.
Since $K_{U_{t_i}}+A_{{t_i}}$ is ample, $(U_{t_i}, A_{t_i})$ is the lc model of 
$(U'_{t_i}, E'_{t_i}+A_{t_i}')$. In particular, 
$(U'_{t_i}, E'_{t_i}+A_{t_i}')$ has a good minimal model which is a dlt model of $(U_{t_i},A_{t_i})$.
  
By \cite[Theorem 1.2]{HMX18}, 
$$
(U', \Delta':=E'+A')
$$ 
has a good minimal model over $T$.
Let $U'\dasharrow U^c$ be the corresponding lc model, and let $\Delta^c$ 
be the pushforward of $\Delta'$. 
Then the log fibre $(U^c_{t_i},\Delta^c_{t_i})$ over $t_i$ is isomorphic to $(U_{t_i},A_{t_i})$. 
In particular, the Cartier index of $K_{U_{t_i}}+A_{t_i}$ is bounded which in turn implies 
that the Cartier index of $K_{U_{t_i}}$ is bounded. Since $U_{t_i}\simeq U_i$, the 
Cartier index of $K_{U_i}$ is bounded, a contradiction.
\bigskip

It is worth elaborating more on the above example. By construction, $S_i\to V_i$ factors through $h_i$. The induced morphism $U_i\to V_i$ 
contracts $Z_i$ only. Since $-K_{U_i}$ is ample, $U_i\to V_i$ is a flipping contraction. 
Recall that $U_i\to V_i$ is defined by the 
big and base point free divisor $H_i$. We claim that the $V_i$ belong to a bounded family. 
Using 
$$
{\pi_i}_*\mathcal{O}_{S_i}(H_i)\simeq \mathcal{O}_{X_i}\oplus \mathcal{O}_{X_i}(A_i)
$$
and its symmetric powers we can show that $\vol(H_i)$ is bounded. Moreover, 
$|H_i|$ defines a birational map: indeed, assume $v,w$ are general closed points on $S_i$; if $\pi_i(v)\neq \pi_i(w)$, then using the above isomorphism of sheaves 
we can find a section in $H^0(S_i,H_i)$ vanishing at $v$ but not at $w$, and vice versa; if $\pi(v)=\pi(w)$, then using the fact that $|H_i|$ is free and that $\deg H_i|_{F}=1$, where $F$ is the fibre of $\pi_i$ containing $v,w$, we can again find a section in $H^0(S_i,H_i)$ vanishing at $v$ but not at $w$, and vice versa. 
Therefore, the $V_i$ belong to a bounded family.  

Let $H_{V_i}$ be the pushforward of $H_i$. 
We can find $0\le L_{V_i}\sim 7H_{V_i}$, so that $(S_i,X_i+L_i)$ is lc and log smooth where 
$L_i$ is the pullback of $L_{V_i}$. Now let $U_i^+\to V_i$ be the flip of $U_i\to V_i$. 
Then $(U_i^+,L_{U_i^+})$ is the ample model of $(S_i,X_i+L_i)$. Since the latter 
form a bounded family, the $(U_i^+,L_{U_i^+})$ also form a bounded family. 
Note that since $U_i$ is smooth outside $Z_i$, $U_i^+$ is klt.  
We have then an example of an unbounded family $U_i$ which becomes klt and bounded after a flip.


\section{Boundedness of generalised pairs on bases of fibrations}

In this section we prove our main result on boundedness of generalised pairs that are induced by adjunction for fibrations as in Theorem \ref{t-bnd-base-adjunction}. We will first need to make some preparations.

\subsection{Preparations}
We start with a reduction from the lc case to the dlt case. 

\begin{lem}\label{l-bnd-base-adjunction-lc-to-dlt} 
Let $d\in \mathbb N$, $\Phi\subset \Q^{\ge 0}$ be a DCC set, and $u\in \Q^{>0}$.  
Then there exists a finite set $\Pi\subset \Q^{>0}$ satisfying the following.
Consider the set of projective pairs $(X,B)$ and divisors $A\ge 0$ such that 
\begin{itemize}
\item $(X,B)$ is lc of dimension $d$,
\item $K_X+B$ is semi-ample defining a contraction $f\colon X\to Z$,
\item over the generic point $\eta_Z$ of $Z$: $A$ is ample and contains no non-klt centre of $(X,B)$,
\item the coefficients of $B$ and the horizontal coefficients of $A$ are in $\Phi$, 
\item $\vol(A|_F)=u$ for general fibres $F$ of $f$. 
\end{itemize}
Then there exist a crepant model $(X',B')$ of $(X,B)$ and an integral divisor $A'\ge 0$ 
on $X'$ such that 
\begin{itemize}
    \item $\vol(A'|_{F'})\in \Pi$ for the general fibres $F'$ of $X'\to Z$, 
\item the coefficients of $B'$ belong to $\Phi\cup \{1\}$, 
and 
\item over $\eta_Z$ we have: $(X',B')$ is dlt, $A'$ is ample and contains no non-klt centre of $(X',B')$.
\end{itemize}
\end{lem}
\begin{proof}
We can fix the relative dimension $r=\dim X-\dim Z$. If $r=0$, then the statement holds by simply taking $(X',B')$ to be a $\Q$-factorial dlt model of $(X,B)$ and $A'$ be any effective integral divisor. So assume $r>0$. 
Let $F$ be a general fibre of $f$ and $B_F=B|_F$ and $A_F=A|_F$.
Let $\delta$ be the minimum positive number in $\Phi$. Then the coefficients of $A_F$ are $\geq \delta$.
By \cite[Theorem 6.4]{B20}, there exists a positive rational number $s>0$ depending only on $d,\Phi,u$ such that $(F,B_F+sA_F)$ is log canonical and hence a stable pair with volume
$\vol (K_F+B_F+sA_F)=s^ru$.
By \cite[Theorem 1.1]{HMX18}, it follows that the pairs $(F,B_F+sA_F)$ are bounded and the horizontal coefficients of $B$ and $A$ are in a finite subset $\Phi _0\subset \Phi$. 
It follows that there is a fixed integer $m\in \mathbb{N}$ so that $H:=m(K_X+B+sA)$ is very ample over $\eta_Z$ and the volume $v:=\vol(H|_F)$ is fixed. 

Clearly there exist a crepant model $(X',B')$ of $(X,B)$ and an integral divisor $A'\ge 0$ 
on $X'$ such that $A'$ is very ample over $\eta_Z$, the coefficients of $B'$ belong to $\Phi\cup \{1\}$, and over $\eta_Z$ we have: $(X',B')$ is dlt, $A'$ contains no non-klt centre of $(X',B')$. 
Choose such $(X',B'),A'$ with minimal $\vol(A'|_{F'})$ where $F'$ is a general fibre of $X'\to Z$. 

It is enough to show that $\vol(A'|_{F'})$ is bounded from above. Assume not, and  choose a 
sequence $\Xi=\{(X,B),A\to Z\}$ so that for the corresponding sequence $(X',B'),A'$ the volumes $\vol(A'|_{F'})$ are not bounded from above. Perhaps after replacing the sequence with a subsequence and applying Lemma \ref{l-universal-family}, we can assume that there exists a semi-ample family 
$(\mathcal{V},\mathcal{B})\to V$ such that for each family $(X,B),A\to Z$ in $\Xi$, over the generic point $\eta_Z$ the family $(X,B)\to Z$ is the pullback of 
$(\mathcal{V},\mathcal{B})\to V$ via some rational map $Z\bir V$, and the images of the $\eta_Z$ form a dense set in $V$. 
 
By definition of semi-ample families, $(\mathcal{V},\mathcal{B})$ is lc.  
Pick a dlt model $(\mathcal{V}'',\mathcal{B}'')$ of $(\mathcal{V},\mathcal{B})$ together with a 
general very ample$/V$ Cartier divisor $\mathcal{A}''\ge 0$ not containing any non-klt centre of 
$(\mathcal{V}'',\mathcal{B}'')$. Then we can assume that pulling back the family $(\mathcal{V}'',\mathcal{B}''), \mathcal{A}''\to V$ over some appropriate open subset $S\subset Z$ via $Z\bir V$ we get a semi-ample family $(U'',B_{U''})\to S$ and a very ample$/S$ divisor $A_{U''}\ge 0$ such that $(U'',B_{U''})$ is a dlt model of $(X,B)$ over $S$ and $A_{U''}$ does not contain any non-klt centre of $(U'',B_{U''})$.

Take a compactification $X'''$ of $U''$ over $X$ and take a log resolution $W\to X'''$ so that it is an isomorphism over the generic point of the non-klt centres of $(U'',B_{U''})$ (this is possible by the dlt condition). Let $U\subset X$ be the image of $U''$. Let $E_1$ be the sum of the exceptional$/X$ prime divisors on $W$ whose generic point maps into $U$, and let $E_2$ be the sum of the other exceptional$/X$ prime divisors on $W$. 
Let $B^\sim$ be the birational transform of $B$ and write $B^\sim=B^\sim_1+B^\sim_2$ where the generic point of each component of $B^\sim_1$ maps into $U$ but the generic point of each component of $B^\sim_2$ maps outside $U$. Let 
$$
C_W=E_1+B^\sim_1+eE_2+eB^\sim_2
$$ 
where $e<1$ is a rational number sufficiently close to $1$. Let $A_{U''}^\sim$ be the birational transform of $A_{U''}$ on $W$. 

Pick a small rational number $t>0$. Then, perhaps after removing some vertical$/Z$ components of $A_{U''}^\sim$ and after some blowups of $W$, we can assume that the pair $(W,C_W+tA_{U''}^\sim)$ is lc and that the generic point of each of its non-klt centres  maps into $U$. Over $U$, $(U'',B_{U''}+tA_{U''})$ is an lc model of $(W,C_W+tA_{U''}^\sim)$, hence $(W,C_W+tA_{U''}^\sim)$ has a good minimal model over $U$ \cite{B12}. Thus $(W,C_W+tA_{U''}^\sim)$ has a good minimal model $(X'',C''+tA'')$ over $X$, by \cite{HX13}, which we can assume to be $\Q$-factorial. Over $U$, $(X'',C''+tA'')$ is a dlt model of $(U'',B_{U''}+tA_{U''})$, so the map $X''\bir X'''$ is a morphism over $U$. Since $t$ is small, 
we can assume that the exceptional divisors of $W\bir X''$ are independent of $t$. In particular, over $X$, $K_{X''}+C''$ is a limit of movable$/X$ divisors.  Then by the general negativity lemma \cite[Lemma 3.3]{B12}, denoting $X''\to X$ by $\phi$ we have $\phi^*(K_X+B)-(K_{X''}+C'')\ge 0$. Thus every exceptional prime divisor of $\phi$ has log discrepancy zero with respect to $(X,B)$ as the coefficient of such divisors in $C''$ is $\ge e$ and $e$ is sufficiently close to $1$. Let $K_{X''}+B''=\phi^*(K_X+B)$. Then we see that 
$(X'',B'')$ is a crepant model of $(X,B)$, the coefficients of $B''$ are in $\Phi\cup \{1\}$, and over $\eta_Z$, $A''\ge 0$ does not contain any non-klt centre of $(X'',B'')$. 

Now replace $X''$ with the lc model of $(X'',C''+tA'')$ over $X$. Then over $U$, $(X'',C''+tA'')$ is just $(U'',B_{U''}+tA_{U''})$. In particular,  over $\eta_Z$, $(X'',B'')$ is dlt and $A''\ge 0$ is very ample. 
By construction, $\vol(A''|_{F''})$ belongs to a fixed finite set $\Pi$ where $F''$ is a general fibre of $X''\to Z$. This contradicts the assumption that $\vol(A'|_{F'})$ is not bounded from above where $F'$ is a general fibre of $X'\to Z$ and the fact that $\vol(A''|_{F''})\ge \vol(A'|_{F'})$ by our choice of $(X',B'),A'$.

\end{proof}

\subsection{Adjunction for generic dlt pairs}

Next we treat adjunction on non-klt centres for generic dlt pairs.

\begin{lem}\label{l-generic-dlt-adjunction}
Let $d\in \mathbb N$ and let $\Phi\subset \Q^{\ge 0}$ be a DCC set. Then there is a 
DCC set $\Psi$ satisfying the following.
Let $(X,B)$ be an lc pair of dimension $d$ where the coefficients of $B$ are in $\Phi$. 
Assume $V$ is a non-klt centre of $(X,B)$ and that $(X,B)$ is dlt near the 
generic point $\eta_V$. Let $V^\nu$ be the normalisation of $V$. 
Then we have an adjunction formula $K_{V^\nu}+B_{V^\nu}=(K_X+B)|_{V^\nu}$ where the 
coefficients of $B_{V^\nu}$ belong to $\Psi$. 
\end{lem}
\begin{proof}
Since $(X,B)$ is dlt near $\eta_V$, there is a log resolution $\phi\colon W\to X$ which is an isomorphism over $\eta_V$. Let $B_W$ be the sum of the reduced exceptional divisor of $\phi$ and the birational transform of $B$. Running an MMP on $K_W+B_W$ over $X$ with scaling of an ample divisor ends with a model $X'$ such that $(X',B')$ is a dlt model of $(X,B)$, where $B'$ is the pushforward of $B_W$, and so that $X'\to X$ is an isomorphism over $\eta_V$. 
In particular, $V$ has a birational transform $V'$ on $X'$. It is then well-known that 
we have an adjunction formula $K_{V'}+B_{V'}=(K_{X'}+B')|_{V'}$ where the 
coefficients of $B_{V'}$ belong to a fixed DCC set $\Psi$. Moreover, $V'$ is normal, 
so $V'\to V$ factors through $V^\nu$. Now let $B_{V^\nu}$ be the pushforward of $B_{V'}$. 

\end{proof}

\subsection{Uniform Galois covers}
We will now discuss Galois covers in families. 

\begin{lem}\label{l-finite-cover}
Let $d,r \in \mathbb{N}$. Consider commutative diagrams 
$$
\xymatrix{
& {S}'\ar[ld]^\alpha\ar[rd]^\beta &\\
{S}\ar[rd]^\pi && {Z}' \ar@{-->}[ld]^\phi\\
& Z&}
$$ 
of normal projective varieties of dimension $d$ where $\alpha$ and $\phi$ are birational and $\pi$ is finite. Assume $H_{S},H_{Z'}$ are very ample divisors on $S,Z'$ with $(\alpha^*H_S+\beta^*H_{Z'})^d\le r$. Then there is a very ample 
divisor $H_Z$ on $Z$ such that $(\alpha^*H_S+\beta^*H_{Z'}+\alpha^*\pi^*H_Z)^d$ 
is bounded from above, in particular, such $Z$ belong to a bounded family.
\end{lem}
\begin{proof}
Consider the induced rational map $S\bir Z'$. We can replace $S'$ with the normalisation of the graph $\Gamma$ of $S\bir Z'$. Then $\alpha^*H_S+\beta^*H_{Z'}$ is ample, base point free and defines a birational morphism $\psi\colon S' \to \Gamma\subset \mathbb P^n$ such that 
$\psi^*R|_\Gamma=\alpha^*H_S+\beta^*H_{Z'}$, where $R$ is a hyperplane, and $S' \to \Gamma$ is the normalization.
Since $(R|_\Gamma)^d=(\alpha^*H_S+\beta^*H_{Z'})^d$ is bounded, it follows that $\Gamma$ is bounded and hence $S'$ belongs to a bounded family. Therefore, there exist projective morphisms    
$\mathcal{S}'\to \mathcal{Z}' \to T$
of normal varieties of finite type where $\mathcal{S}'\to \mathcal{Z}'$ is generically finite, 
so that each  $S' \to Z'$ as above is isomorphic to $\mathcal{S}_t'\to \mathcal{Z}_t'$ for some closed point $t\in T$. We may further assume that the set of such points $t\in T$ is dense in each component of $T$.

There is a finite morphism $\mathcal{V}'\to \mathcal{S}'$ from a normal variety so that the induced morphism $\mathcal{V}'\to \mathcal{Z}'$ is Galois over the generic point of $\mathcal{Z}'$, i.e. the corresponding extension of function fields is Galois, say with Galois group $G$. 
Then over some non-empty open subset $\mathcal{U}'$ of $\mathcal{Z}'$ we have: $\mathcal{V}'\to \mathcal{Z}'$ is a Galois finite \'etale cover and $G$ acts on $\mathcal{V}'$ with quotient being $\mathcal{Z}'$. 

Pick $S\leftarrow S' \to Z'$ as above where we assume $S'$ is the normalised graph of $S\bir Z'$ and that $S'\to Z'$ is isomorphic to $\mathcal{S}_t'\to \mathcal{Z}_t'$ for some closed point $t\in T$.
It is enough to consider the case when  $\eta_{\mathcal{Z}_t'}\in \mathcal{U}'$ and when $\mathcal{V}_t'$ is a normal variety, by stratifying $T$ and applying Noetherian induction. 
Then over $\mathcal{Z}_t'\cap \mathcal{U}'$ we have: 
$$
V':=\mathcal{V}_t'\to \mathcal{Z}_t'=Z'
$$ 
is a Galois finite \'etale cover with Galois group $G$, and $G$ acts on $\mathcal{V}_t'$ with quotient being $\mathcal{Z}_t'\cap \mathcal{U}'$. 

By construction, $V'\to Z'$ factors through $S'\to Z'$.
Now let $V$ be the closure of $S$ in the function field $k(V')$. Then $V\to S$ is a finite morphism which induces a finite morphism $V\to Z$. Since $Z'\bir Z$ is birational,  $V\to Z$  is generically Galois with Galois group $G$. In fact $V$ is just the closure of $Z$ in $k(V')$. 
In particular, $G$ acts on $V$ with quotient being $Z$. 

Denote $V\to S$ by $\gamma$. Put $L=\sum_{g\in G}g^*(\gamma^*H_S)$. Then $L$ is a $G$-invariant divisor, so $L=\gamma^*\pi^*N$ for some ample Cartier divisor $N$ on $Z$. We show that $|N|$ is a base point free linear system which defines a birational quasi-finite map $Z\to \bar Z$. Thus $Z\to \bar Z$ is the normalization of $\bar Z$ and $\bar N$ is a very ample line bundle on $\bar Z$ with $N=\nu ^* \bar N$. Pick a closed point $z\in Z$. Since $H_S$ is very ample, we can choose $H_S$ in its linear system so that it is effective and that it does not contain any element of $\pi^{-1}\{z\}$. Then $\gamma^*H_S$ does not contain any element of $\gamma^{-1}\pi^{-1}\{z\}$. Since $G$ does not map any point in $V\setminus \gamma^{-1}\pi^{-1}\{z\}$ into $\gamma^{-1}\pi^{-1}\{z\}$,  $g^*(\gamma^*H_S)$ does not intersect $\gamma^{-1}\pi^{-1}\{z\}$ for any $g$, hence 
$L$ does not intersect $\gamma^{-1}\pi^{-1}\{z\}$ as well. Therefore, we can choose $N$ in its linear system so that it does not contain $z$. We can similarly show that if $z_1,z_2$ are general closed points in $Z$, then there is a section of $N$ which vanishes at $z_1$ but does not vanish at $z_2$, so $|N|$ defines a birational map. 

Now we show that $\vol(N)$ is bounded from above. It is enough to show that $\vol(L)$ is bounded. Since $(\alpha^*H_S+\beta^*H_{Z'})^d$ is bounded, there exists a fixed number $l\in \mathbb{N}$ such that $l\beta^*H_{Z'}-\alpha^*H_S$ is big: indeed taking $H_S,H_{Z'}$ general in their linear systems we see that $(S',\alpha^*H_S+\beta^*H_{Z'})$ belongs to a bounded family from which we can deduce the claim. Thus the pullback of $l\beta^*H_{Z'}-\alpha^*H_S$ to $V'$ is big. On the other hand, let $\bar{V}'$ be the closure of $Z'$ in $k(V')$. Then $G$ acts on $\bar{V}'$ with quotient $\bar{V'}/G$ being $Z'$. Moreover, $\bar{V}'\to Z'$ is the finite part of the Stein factorisation of ${V}'\to Z'$, and $V\to Z$ is the finite part of the Stein factorisation of $\bar{V}'\to Z$. Therefore, if $\lambda,\rho$ denote $\bar{V}'\to S$ and $\bar{V}'\to Z'$ respectively, then $l\rho^*H_{Z'}-\lambda^*H_S$ is big because $V'\to \bar{V}'$ is a birational morphism.
So we see that 
$$
|G|l\rho^*H_{Z'}-\sum_{g\in G}g^*(\lambda^*H_S)
$$ 
is big. In particular, 
the volume of $\sum_{g\in G}g^*(\lambda^*H_S)$ is bounded from above which in turn implies volume of $L$ is bounded from above. 

Finally, the above arguments show that $(\bar Z,\bar N)$ belong to a bounded family and hence so do $(Z,N)$. Thus a bounded multiple $H_Z$ of $N$ is very ample. Arguing similarly to the previous paragraph by working on $\bar{V}'$ and comparing with $\rho^*H_{Z'}$ one shows that the volume of $\alpha^*H_S+\beta^*H_{Z'}+\alpha^*\pi^*H_Z$ is also bounded. 

\end{proof}

\subsection{Adjunction for fibrations}
We now consider adjunction for fibrations in the setting of Theorem \ref{t-bnd-base-adjunction}.

\begin{lem}\label{l-adj-fib-bnd-bir-base}
Let $(X,B),A\to Z$ be as in Theorem \ref{t-bnd-base-adjunction}.  

(i)
There exist $2\le q\in \mathbb N$ and a DCC set $\Psi$ depending on $d$, $\Phi$, and $u$ such that we can write an adjunction formula 
\[ q(K_X+B)\sim qf^*(K_Z+B_Z+M_Z)\]
where $qM_{Z''}$ is Cartier for $M_{Z''}$  the moduli part on any sufficiently high resolution $\nu:Z''\to Z$. Moreover, 
$$
(Z,B_Z+M_Z)\in \mathcal{F}_{glc}(e,\Psi,v)
$$ 
for some $0\le e \le d$. 

(ii) 
There exists a bounded set of couples $\mathcal P$ such that for each $(Z,B_Z+M_Z)$ as above, there exists a
log smooth couple $(Z',\Sigma' )\in \mathcal P$ and a birational map $\eta: Z'\dasharrow Z$ such that
\begin{enumerate}
    \item $B_{Z'}\leq \Sigma'$ where $B_{Z'}:={\rm Ex}(\eta)+\eta ^{-1}_*B_Z$ is the sum of the reduced exceptional divisor for $Z'\dasharrow Z$ plus the strict transform of $B_Z$,
    \item the moduli b-divisor descends to $M_{Z'}$ on $Z'$,
    \item $\vol(K_{Z'}+B_{Z'}+M_{Z'})=v$, 
    \item $a(P,Z,B_Z+M_Z)\geq a(P,Z',B_{Z'}+M_{Z'})$ for any prime divisor $P$ over $Z$, and
    \item $Z\bir Z'$ does not contract any divisor.
\end{enumerate}
\end{lem}
\begin{proof}
By Lemma \ref{l-bnd-base-adjunction-lc-to-dlt}, we may replace $(X,B),A$ so that we can assume that $A|_F$ is integral. By \cite[Lemma 7.4]{B21}, there exist $2\le p,q\in \mathbb N$ (depending on $d$, $\Phi$, 
and $u$) such that we can write the adjunction formula 
\[ q(K_X+B)\sim qf^*(K_Z+B_Z+M_Z)\]
where $pM_{Z''}$ is Cartier for $M_{Z''}$  the moduli part on any sufficiently high resolution $\nu:Z''\to Z$. Replacing $p,q$ by $pq$ we can assume $qM_{Z''}$ is Cartier. 

Moreover, by \cite{HMX14}, the coefficients of $B_Z$ belong to a DCC set $\Psi$ depending only on $d, \Phi$. Thus replacing $\Psi$ with $\Psi \cup \frac{1}{q}\mathbb{Z}$, it follows that 
$$
(Z,B_Z+M_Z)\in \mathcal{F}_{glc}(e,\Psi,v)
$$ 
where $0\leq e\leq d$. This proves (i).

By \cite[Proposition 5.2]{B21},  there
exists a bounded set of couples $\mathcal P$ such that for each $(Z,B_Z+M_Z)$ as above, there exists a log smooth couple $(Z',\Sigma' )\in \mathcal P$ and a birational map $\eta: Z'\dasharrow Z$ satisfying (1)-(3). To see (4), let $\mu :\bar Z \to Z'$ and $\nu:\bar Z \to Z$ be a common resolution and note that by construction, 
$$
F:=\mu ^*(K_{Z'}+B_{Z'}+M_{Z'})-\nu ^*(K_{Z}+B_{Z}+M_{Z})
$$ 
is anti-nef over $Z'$ and $\mu _* F\geq 0$. By the negativity lemma, $F\geq 0$ and thus (4) holds.
(5) is then an immediate consequence of \cite[Lemma 2.17 ]{B21}.

\end{proof}

\subsection{Reduction to the generic klt case}
We reduce Theorem \ref{t-bnd-base-adjunction} to the generic klt case. This will allow us to eventually reduce Theorem \ref{t-bnd-base-adjunction} to some statement about klt log Calabi-Yau fibrations.

\begin{lem}\label{l-bnd-base-adjunction-lc-to-klt}
Assume that Theorem \ref{t-bnd-base-adjunction} holds when $(X,B)$ is klt over $\eta_Z$. 
Then the theorem holds in general. 
\end{lem}
\begin{proof}
\emph{Step 1.} 
Suppose that $(X,B)$ is not klt over $\eta_Z$. Applying Lemma \ref{l-bnd-base-adjunction-lc-to-dlt}, we can assume that $(X,B)$ is dlt over $\eta_Z$. Let $q,\Psi$ and the adjunction formula 
\[ q(K_X+B)\sim qf^*(K_Z+B_Z+M_Z)\]
be as in Lemma \ref{l-adj-fib-bnd-bir-base}.
Let $V$ be a non-klt centre of $(X,B)$ dominating $Z$, 
which is minimal among such dominating non-klt centres. Let $V^\nu\to V$ be the normalisation. Since $(X,B)$ is dlt near the generic point $\eta_V$, by Lemma \ref{l-generic-dlt-adjunction}, we can  
define $K_{V^\nu}+B_{V^\nu}=(K_X+B)|_{V^\nu}$ where the coefficients of $B_{V^\nu}$ are in a fixed DCC set. Let $V^\nu\to S\overset{\pi}\to Z$ be the Stein factorisation 
of $V^\nu\to Z$.  Then $(V^\nu,B_{V^\nu})$ is an lc pair which is klt over $\eta _Z$ and hence klt over $\eta _S$, the generic point of $S$. Clearly, $K_{V^\nu}+B_{V^\nu}$ is semi-ample with Iitaka volume $\deg(\pi)v$ and its associated contraction is just $V^\nu\to S$.\\

\emph{Step 2.}
If $F$ is a general fibre of $X\to Z$ and $K_F+B_F=(K_X+B)|_F$ and $A_F=A|_F$, then $(F,B_F+sA_F)$ is dlt and it belongs to a bounded family for some fixed $s\in \Q^{>0}$, as observed in the proof of Lemma \ref{l-bnd-base-adjunction-lc-to-dlt}, and the coefficients of $B_F+sA_F$ take finitely many possible values. In particular, the non-klt centres of $(F,B_F)$ belong to a bounded family, and the number of such centres is bounded. Now each irreducible component of $F\cap V$ is a non-klt centre of $(F,B_F)$, 
hence the number of connected components of $F\cap V$ is bounded. But $F\cap V$ is the fibre of $V\to Z$ over $f(F)$, so the number of connected components of $F\cap V$ is just the degree of $\pi\colon S\to Z$, so this degree is bounded. In particular, 
this means that 
$$
\vol(\pi^*(K_Z+B_Z+M_Z))=(\deg \pi)v
$$ 
takes finitely many possible values. Moreover, over $\eta_S$, $A|_{V^\nu}$ is ample containing no non-klt centre of $(V^\nu,B_{V^\nu})$. In addition, if $G$ is a general fibre of $V^\nu\to S$, then $G$ is an irreducible component of $F\cap V$, hence it is a non-klt centre of the bounded pair $(F,B_F+sA_F)$, so we can see that the coefficients of $A|_G$ and $\vol(A|_G)$ each take finitely many possible values (note that $(X,B)$ is dlt over $\eta_Z$, so $V=V^\nu$ over $\eta_Z$). Thus the horizontal coefficients of $A|_{V^\nu}$ also take finitely many possible values.

Therefore, we can assume that $(V^\nu,B_{V^\nu})\to S$ satisfies the hypothesis of Theorem \ref{t-bnd-base-adjunction} for certain fixed data $d',\Phi',u',v'$ in place of $d,\Phi,u,v$, and moreover $(V^\nu,B_{V^\nu})$ is klt over $\eta _S$.   
Thus, we may assume that $S$ is bounded by assumption. That is, there is a 
very ample divisor $H_S$ with bounded $H_S^{e}$ and $\pi^*(K_Z+B_Z+M_Z)\cdot H_S^{e-1}$ 
where $e=\dim S=\dim Z$.\\ 

\emph{Step 3.}
Recall the bounded birational model $K_{Z'}+B_{Z'}+M_{Z'}$ given by Lemma \ref{l-adj-fib-bnd-bir-base} (ii). Arguing as in 
Step 4 of the proof of \cite[Theorem 1.3]{B21} and perhaps after replacing $q$ with a 
multiple we can find a family $(\mathcal{Z}', \mathcal{B}'+\mathcal{M}')\to T$ where $T$ is quasi-projective, $(\mathcal{Z}', \mathcal{B}')$ is log smooth over $T$, 
for each $(Z',B_{Z'}+M_{Z'})$ as above there is  a closed point $t\in T$ so that 
$(Z',B_{Z'})$ is isomorphic to $(\mathcal{Z}_t', \mathcal{B}_t')$ and $qM'\sim q\mathcal{M}_t'$, and the set of such points $t$ is dense in each component of $T$. Note that we are not assuming 
$(\mathcal{Z}', \mathcal{B}'+\mathcal{M}')$ to be a generalised pair, i.e. $\mathcal{M}'$ is not necessarily nef over $T$.

By \cite[Theorem 1.3]{BZ16}, there exists an integer $m>0$, depending only on $e,\Psi, v$ such that $|m(K_{Z'}+B_{Z'}+M_{Z'})|$ defines a birational map for each $(Z',B_{Z'}+M_{Z'})$ as above. 
Note that $h^0(m(K_{\mathcal{Z}'}+{\mathcal{B}'}+\mathcal{M}'))$ is locally constant on some dense open subset $T^0\subset T$.
Thus, possibly shrinking $T^0$, by base change of cohomology, we may assume that 
$$
H^0(m(K_{\mathcal{Z}'}+{\mathcal{B}'}+\mathcal{M}'))\to H^0(m(K_{\mathcal{Z}'_t}+{\mathcal{B}'_t}+\mathcal{M}'_t))
$$
is surjective for every closed point $t\in T^0$ which implies that $K_{\mathcal{Z}'}+{\mathcal{B}'}+\mathcal{M}'$ is big over each component of $T$ because 
$|m(K_{\mathcal{Z}'_t}+{\mathcal{B}'_t}+\mathcal{M}'_t)|$ defines a  birational map for a dense subset of points $t\in T$.

Fix a divisor $\mathcal H' $ on $\mathcal{Z}'$ which is very ample over $T$.
Shrinking $T^0$ and replacing $m$ by a multiple, we may assume that $m(K_{\mathcal{Z}'_t}+{\mathcal{B}'_t}+\mathcal{M}'_t)-\mathcal H' _t$ is big for every $t\in T^0$. By Notherian induction, decomposing $T$ in to a finite union of locally closed subsets, we may assume that $m(K_{\mathcal{Z}'_t}+{\mathcal{B}'_t}+\mathcal{M}'_t)-\mathcal H' _t$ is big for every $t\in T$.
In particular, there exists a very ample divisor $H_{Z'}$ such that 
$$
m(K_{Z'}+B_{Z'}+M_{Z'})-H_{Z'}
$$ 
is big. Similarly we can assume that 
$$
-(K_{Z'}+B_{Z'}+M_{Z'})+nH_{Z'}
$$ 
is big for some fixed $n\in\mathbb{N}$.

Let $H_Z$ be the pushforward of $H_{Z'}$ to $Z$. Then $m(K_{Z}+B_{Z}+M_{Z})-H_{Z}$ is also big. Thus 
$$
\pi^*H_Z\cdot H_S^{e-1}\leq m\pi^*(K_{Z}+B_{Z}+M_{Z})\cdot H_S^{e-1}
$$ 
is bounded. Note that $\pi^*H_Z$ makes sense even if $H_Z$ is not $\Q$-Cartier because $\pi$ is finite.\\ 

\emph{Step 4.}
Consider the induced rational map $S\bir Z'$. Let $S'$ be the normalisation of the graph of this map 
and $\alpha\colon S'\to S$ and $\beta \colon S'\to Z'$ the induced morphisms. The divisor $H_Z$ is $\Q$-Cartier on the complement $J$ of a codimension two subset of $Z$, so $\pi^*H_Z$ is $\Q$-Cartier on the complement $J'$ of a codimension two subset of $Z'$. 
By the negativity lemma applied over $J'$, $\alpha ^*\pi^*H_Z-\beta ^* H_{Z'}\geq 0$ holds over $J'$, and so 
$$
\alpha^*H_S^{e-1}\cdot \beta^*H_{Z'}\leq \alpha^*H_S^{e-1}\cdot \alpha ^*\pi^*H_Z
$$ 
is bounded. 
Since $\alpha^*H_S$ and $\beta^*H_{Z'}$ are nef, we have inequalities (cf. \cite[Corollary 1.6.3]{Laz04})
\[ (\alpha^*H_S^{e-1}\cdot \beta^*H_{Z'})^k\geq (\alpha^*H_S^{e-k}\cdot \beta^*H_{Z'}^k)\cdot (\alpha^*H_S^e)^{k-1}\]
so that $\alpha^*H_S^{e-k}\cdot \beta^*H_{Z'}^k$ is bounded for $0\leq k\leq e$.
Thus $(\alpha^*H_S+\beta^*H_{Z'})^e$ is bounded. 

Now  $Z$ is bounded by Lemma \ref{l-finite-cover}. More precisely, there is a very ample 
divisor $L_Z$ on $Z$ such that $(\alpha^*H_S+\beta^*H_{Z'}+\alpha^*\pi^*L_{Z})^e$ 
is bounded from above. In particular, $L_Z^e$ and $L_Z^{e-1}\cdot K_Z$ are  bounded. Since 
$$
-(K_{Z'}+B_{Z'}+M_{Z'})+nH_{Z'}
$$
is big and 
$$
\beta^*(K_{Z'}+B_{Z'}+M_{Z'})\ge \alpha^*\pi^*(K_Z+B_Z+M_Z),
$$
it follows that 
\[L_Z^{e-1}\cdot (K_Z+B_Z+M_Z)\le \alpha^*\pi^*L_Z^{e-1}\cdot \beta^*(nH_{Z'})\]
and thus $L_Z^{e-1}\cdot (K_Z+B_Z+M_Z)$ is bounded from above. Since  $L_Z^{e-1}\cdot K_Z$ is bounded,  $L_Z^{e-1}\cdot B_Z$ and $L_Z^{e-1}\cdot M_Z$ are bounded. Now rename $L_Z$ to $H$ to be as in the statement of the theorem.

\end{proof}

\subsection{Boundedness of generalised lc models with bounded singularities and volume}

\begin{lem}\label{l-bnd-e-lc-vol<v}
Let $d\in \mathbb{N}$, $\Phi\subset \mathbb{Q}^{\ge 0}$ be a finite set, and $\epsilon,v\in \mathbb{Q}^{>0}$. 
Then the set of 
$$
(X,B+M)\in \mathcal{F}_{gklt}(d,\Phi,<v)
$$ 
such that $(X,B+M)$ is generalised $\epsilon$-lc, forms a bounded family.
\end{lem}
\begin{proof}
Decreasing $\epsilon$ we can assume it is rational. 
By \cite[Theorem 1.2]{B21}, there
exists a bounded set of log smooth couples $\mathcal P$ such that for each $(X,B+M)$ as in the lemma there exists a log smooth birational model 
$(\overline{X},\overline{B}+\overline{M})$ 
where $\overline{B}=(1-\epsilon)\overline{E}+B^\sim$, $\overline{E}$ is the reduced exceptional divisor of $\overline{X}\dasharrow X$ and $B^\sim$ is 
the strict transform of $B$, the nef part $M'$ descends to $\overline{X}$ as $\overline{M}$, and $\overline{B}\leq \overline{\Sigma}$ for some couple $(\overline{X},\overline{\Sigma})\in \mathcal P$. 
Note that since $\Phi$ is finite, we can assume that $pM'$ (and hence also $p\overline{M}$) is 
Cartier for some fixed $p\in \mathbb N$. Expanding $\Phi$, we can assume $1-\epsilon$ belongs to 
$\Phi$.

Since $K_X+B+M$ is ample and $(X,B+M)$ is generalised $\epsilon$-lc, for any prime divisor $D$ over $\overline{X}$, we have 
$$
a(D,X,B+M)\ge a(D,\overline{X},\overline{B}+\overline{M})=a(D,\overline{X},\overline{B}).
$$ 
The above inequality is immediate for any divisor on $\overline{X}$ and follows easily, for any prime divisor $D$ over $\overline{X}$, by a standard application of the negativity lemma.
In particular, if $D$ is a divisor on $X$ which is exceptional over $\overline{X}$, then 
$$
1\ge a(D,X,B+M)\ge a(D,\overline{X},\overline{B}+\overline{M})=a(D,\overline{X},\overline{B}).
$$
Since $(\overline{X},\overline{B})$ is bounded and  $\epsilon$-lc, all such divisors $D\subset X$ can be obtained by a bounded number of smooth blowups, which are toroidal with respect to $(\overline{X},\overline{\Sigma})$. 
Therefore, replacing $\overline{X}$ by an appropriate birational model, we can assume that $X\bir \overline{X}$ does not contract any divisor.

Moreover, by the arguments of the proof of \cite[Theorem 1.3]{B21}, there exist finitely many relatively 
log smooth families $(V,\Delta)\to T$ over smooth varieties and $\Q$-Cartier divisors $N$ and a fixed $n\in\mathbb N$ such that  
$(\overline{X},\overline{B})$ is isomorphic to the log fibre $(V_t,\Delta_t)$ over some closed point $t\in T$, and 
$n\overline{M}\sim nN|_{V_t}$ where we have identified $\overline{X}$ and $V_t$ under the given isomorphism. From now on we can 
restrict ourselves to one of the families $(V,\Delta), N\to T$.

By \cite[Theorem 1.12]{Fil18}, $\vol(K_{V_t}+\Delta_t+N_t)$ is independent of $t$ whenever $N_t=N|_{V_t}$ is nef. 
Thus it is enough to consider the case when $\vol(K_{\overline{X}}+\overline{B}+\overline{M})$ is a fixed positive rational number, say $w$. 
But then since $X\bir \overline{X}$ does not contract any divisor,  
$$
\vol(K_{\overline{X}}+\overline{B}+\overline{M})\le \vol(K_X+B+M).
$$
On the other hand, since 
$a(D,X,B+M)\ge a(D,\overline{X},\overline{B}+\overline{M})$ for any prime divisor $D$ over $\overline{X}$, then 
$$
\vol(K_{\overline{X}}+\overline{B}+\overline{M})\ge \vol(K_X+B+M). 
$$
Thus $\vol(K_X+B+M)=w$, hence  
$$
(X,B+M)\in \mathcal{F}_{gklt}(d,\Phi,w).
$$ 
By \cite[Theorem 1.4]{B21}, $\mathcal{F}_{gklt}(d,\Phi,w)$ is bounded so we are done.

\end{proof}

\subsection{Klt log Calabi-Yau fibrations}
We now consider boundedness of log Calabi-Yau fibrations in the klt case.

\begin{lem}\label{l-delta-lc-fin-coeff-base-adjunction}
Let $d\in \mathbb N$, $\Phi\subset \Q^{\ge 0}$ be a finite set, $u,\epsilon\in \Q^{>0}$. 
Then there exist $q\in \mathbb N$ and $\delta \in \Q^{>0}$ satisfying the following.
Consider projective pairs $(X,B)$, contractions $f\colon X\to Z$, and $\Q$-divisors $A$ on $X$ such that 
\begin{itemize}
\item $(X,B)$ is $\epsilon$-lc of dimension $d$,
\item $K_X+B\sim_\Q 0/Z$,
\item $A$ is big over $Z$,
\item the coefficients of $B$ and the horizontal coefficients of $A$ are in $\Phi$, and
\item $\vol(A|_F)\le u$ for general fibres $F$ of $f$. 
\end{itemize} 
Then we can write an adjunction formula 
$$
q(K_X+B)\sim qf^*(K_Z+B_Z+M_Z)
$$
with $qM_{Z'}$ Cartier and $qB_Z$  integral where $M_{Z'}$ is the moduli divisor on any high resolution $Z'\to Z$. Moreover, $(Z,B_Z+M_Z)$ is generalised $\delta$-lc. 
\end{lem}
\begin{proof}
Removing the vertical components of $A$ we can assume all its components are horizontal, hence all its coefficients are in $\Phi$. Expanding $\Phi$ and replacing $A$ with a multiple we can assume it is integral.
Replacing $X$ with the ample model of $A$ over $Z$, we can assume that $A$ is ample over $Z$. Moreover, if $(F,B_F)$ is a general log fibre of $(X,B)\to Z$ and $A_F=A|_F$, then $(F,\Supp (B_F+A_F))$ belongs to a bounded family by \cite[Corollary 1.6]{B20}, hence we can assume that $\vol(A_F)$ is fixed. 
By \cite[Lemma 7.4]{B21}, there exist $2\le p,q\in \mathbb N$ (depending on $d$, $\Phi$, 
and $u$) such that we can write the adjunction formula 
\[ q(K_X+B)\sim qf^*(K_Z+B_Z+M_Z)\]
where $pM_{Z'}$ is Cartier for $M_{Z'}$  the moduli part on any sufficiently high resolution $\nu:Z'\to Z$. Replacing $p,q$ by $pq$ we can assume $qM_{Z'}$ is Cartier. 

Since $(F,\Supp (B_F+A_F))$ belongs to a bounded family, $\vol(\Supp (B_F)+A_F)$ is bounded from above as $\Phi$ is finite. Then $(Z,B_Z+M_Z)$ is generalised $\delta$-lc for some fixed $\delta>0$ depending only on $d,\epsilon,u$, by \cite[Theorem 1.9]{B18}. 
This implies that multiplicities of the fibres of $f$ over codimension one points are bounded, i.e. for any prime divisor $D$ on $Z$, the coefficients of $f^*D$ over the generic point $\eta_D$ belong to a fixed finite set because $(Z,B_Z+M_Z)$ being generalised $\delta$-lc implies $(X,B+\delta f^*D)$ is lc over $\eta_D$ by definition of adjunction. 

Let $D$ be a component of $B_Z$. 
Since $\Phi$ is finite, there is a fixed integer $r\in \mathbb{N}$ such that  $rq(K_X+B)$ is integral. Moreover, $q(K_Z+M_Z)$ is Cartier near $\eta_D$ because $Z$ is normal so it is smooth on a neighborhood of $\eta_D$ and because both $qK_Z$ and $qM_Z$ are integral.  So over $\eta_D$, the divisor $rqf^*(K_Z+M_Z)$ is integral. Then since
$$
 rq(K_X+B)-rqf^*(K_Z+M_Z)\sim rqf^*B_Z
$$
over $\eta_D$, 
we deduce that  $rqf^*B_Z$ is integral over $\eta_D$. This implies that the coefficient of $D$ in $B_Z$ belongs to a fixed finite set: indeed, let $S$ be a component of $f^*D$ mapping onto $D$; then the coefficient $\mu_Sf^*D$ belongs to a fixed finite set by the previous paragraph; therefore, 
$$
\mu_DB_Z=\mu_Srqf^*B_Z/rq\mu_Sf^*D
$$ 
belongs to a fixed finite set. Thus all the coefficients of 
$B_Z$ belong to a fixed finite set, so replacing $q$ we can ensure $qB_Z$ is integral.

\end{proof}

We apply a recent result of \cite{J22} to prove a birational boundedness statement for klt log Calabi-Yau fibrations. 

\begin{lem}\label{l-bnd-klt-fibs}
Let $d\in \mathbb N$, $\Phi\subset \Q^{\ge 0}$ be a finite set, $u,v,\epsilon\in \Q^{>0}$. 
Consider projective pairs $(X,B)$ and $\Q$-divisors $A$ on $X$ such that 
\begin{itemize}
\item $(X,B)$ is $\epsilon$-lc of dimension $d$,
\item $K_X+B$ is semi-ample defining a contraction $f\colon X\to Z$,
\item $A\ge 0$ is integral and ample over $Z$,
\item the coefficients of $B$ are in $\Phi$, 
\item $\vol(A|_F)\le u$ for general fibres $F$ of $f$, and 
\item $\Ivol(K_X+B)\le v$. 
\end{itemize} 
Then there exist a bounded projective pair $(Y,B_Y)$ over $Z$ and a birational map $X\bir Y/Z$ not contracting any divisor where $K_Y+B_Y$ is the pullback of $K_X+B$ under $Y\bir X$. 
\end{lem}
\begin{proof}
As observed in the proof of Lemma \ref{l-delta-lc-fin-coeff-base-adjunction}, for a general fibre $F$ of $f$, $(F,\Supp (B_F+A_F))$ belongs to a bounded family. 
Using the assumption that $A$ is integral and $\vol(A|_F)\le u$, there are then finitely many possibilities for the horizontal coefficients of $A$. Moreover, by Lemma  \ref{l-delta-lc-fin-coeff-base-adjunction}, there exist fixed $q\in \mathbb N$ and $\delta\in \Q^{>0}$ so that we can write an adjunction formula 
$$
q(K_X+B)\sim qf^*(K_Z+B_Z+M_Z)
$$
with $qM_{Z'}$ Cartier where $M_{Z'}$ is the moduli divisor on any high resolution $Z'\to Z$, and $qB_Z$  integral. Moreover, $(Z,B_Z+M_Z)$ is generalised $\delta$-lc. 

By assumption, 
$$
\vol(K_Z+B_Z+M_Z)={\rm Ivol}(K_X+B)\le v,
$$
so taking $\Psi=\{\frac{i}{q} \mid 0\le i\le q\}$, we see that 
$$
(Z,B_Z+M_Z)\in \mathcal{F}_{gklt}(e,\Psi,<2v),
$$
where $0\le e\le d$. Therefore, by Lemma \ref{l-bnd-e-lc-vol<v}, $(Z,B_Z+M_Z)$ belongs to a bounded family. This implies that the Cartier index of $K_Z+B_Z+M_Z$ is bounded: indeed, by \cite[Theorem 1.3]{BZ16}, there is a fixed $m\in \mathbb{N}$ so that $m(K_Z+B_Z+M_Z)\sim D\ge 0$ where $D$ has integral coefficients. So $(Z,\Supp D)$ is bounded and $D$ has coefficients 
in a fixed finite set. Thus some bounded multiple of $D$ is Cartier, by \cite[Lemma 2.24 or 2.25]{B19}. Note this also shows that the Cartier index of $K_X+B$ is bounded by the above adjunction formula.

The previous paragraph implies that there are finitely many possibilities for $\vol(K_Z+B_Z+M_Z)$. Then applying \cite[Theorem 1.2]{J22} to $(X,B)\to Z$, there exist a bounded projective pair $(V,B_V)$ and a birational map $X\bir V$ such that pullbacks of $K_X+B$ and $K_V+B_V$ agree on any common resolution of $X,V$. In particular, $K_V+B_V$ is semi-ample, hence $V\bir Z$ is a morphism. Since the Cartier index of $K_X+B$ is bounded, the coefficients of $B_V$ belong to a fixed finite set. 
Let $(Y,B_Y)$ be a terminal crepant model of $(V,B_V)$. Then by \cite[Theorem 1.3]{B18}, $(Y,B_Y)$ is bounded. Moreover, $X\bir Y$ does not contract any divisor because for any prime divisor $S$ on $X$ we have 
$$
a(S,Y,B_Y)=a(S,X,B)\le 1.
$$ 
Finally $K_Y+B_Y$ coincides with the pullback of $K_X+B$ under $Y\bir X$. 

\end{proof}

\subsection{Log discrepancies}

\begin{lem}\label{l-ldisc}
Let $(X,B),A\to Z$ be as in Theorem \ref{t-bnd-base-adjunction}. Then 
$$
\{a(D,X,B)\le 1 \mid \mbox{$D$ prime divisor over $X$}\}
$$
is a subset of a fixed finite set depending only on $d,\Phi,u,v$
\end{lem}
\begin{proof}
This follows from the proof of \cite[Lemma 8.2]{B21}. In \cite[Lemma 8.2]{B21}, in addition to our assumptions, it is assumed that $(X,B),A$ is a stable pair meaning that $K_X+B+A$ is ample and that $(X,B+tA)$ is lc for some $t>0$. But the proof does not use ampleness of $K_X+B+A$ (see the remark after the proof of \cite[Lemma 8.2]{B21}) and it uses existence of $t$ only over $\eta_Z$ which is satisfied in our situation. So the same exact proof applies to our situation.

\end{proof}

\subsection{Proof of \ref{t-bnd-base-adjunction}}
We are now ready to prove our main result on boundedness of bases of fibrations.

\begin{proof}[Proof of Theorem \ref{t-bnd-base-adjunction}]
\emph{Step 1.}
By Lemma \ref{l-bnd-base-adjunction-lc-to-klt}, we can assume that $(X,B)$ is klt over $\eta_Z$. 
Moreover, applying Lemma \ref{l-bnd-base-adjunction-lc-to-dlt}, we can assume that $A$ is integral. Recall that by Lemma \ref{l-adj-fib-bnd-bir-base} we have the adjunction formula 
\[ q(K_X+B)\sim qf^*(K_Z+B_Z+M_Z)\] 
with $qM_{Z'}$ Cartier where $M_{Z'}$ is the moduli divisor on any high resolution $Z'\to Z$, and that 
$$
(Z,B_Z+M_Z)\in \mathcal F _{\rm glc}(e,\Psi,v)
$$ 
where $e=\dim Z$ and $q,\Psi$ depend only on $d,\Phi,u$.

By Lemma \ref{l-adj-fib-bnd-bir-base}, we have a bounded birational model $(Z',B_{Z'}+M_{Z'})$ of $(Z,B+M)$ 
so that the moduli divisor descends to $Z'$ as $M_{Z'}$. 
Here $Z\bir Z'$ does not contract any divisor and $B_{Z'}$ is the sum of the birational transform of $B_Z$ and the reduced exceptional divisor of $Z'\bir Z$. 

Also recall from Step 3 of the proof of Lemma \ref{l-bnd-base-adjunction-lc-to-klt} that there is a very ample divisor $H_{Z'}$ such that 
$$
m(K_{Z'}+B_{Z'}+M_{Z'})-H_{Z'}
$$ 
and
$$
-(K_{Z'}+B_{Z'}+M_{Z'})+nH_{Z'}
$$ 
are big for some fixed $m,n\in\mathbb{N}$.\\

\emph{Step 2.} 
Take a log resolution $\phi\colon W\to X$ of $(X,B)$ such that $W\bir Z'$ is a morphism and write $K_W+B_W=\phi^*(K_X+B)$. Let $\Gamma_W=B_W^{\ge 0}$. Then 
$$
K_W+\Gamma_W=\phi^*(K_X+B)+P
$$ 
where $P\ge 0$ is exceptional over $X$. Run an MMP on $K_W+\Gamma_W$ over $Z'$ with scaling of some ample divisor. We reach a model $Y$ on which $K_Y+\Gamma_Y$ is a limit of movable$/Z'$ $\Q$-divisors. Let $U'\subset Z'$ be the largest open set on which $Z'\bir Z$ is an isomorphism. Then $(X,B)$ is a weak lc model of $(W,\Gamma_W)$ over $U'$, so $(W,\Gamma_W)$ has a minimal model over $U'$ \cite[Corollary 3.7]{B12}, hence 
the MMP terminates over $U'$ and $K_Y+\Gamma_Y\sim_\Q 0$ over $U'$ \cite[Theorem 1.9]{B12}. In particular, $P_Y=0$ over $U'$ and $P_Y$ is vertical over $Z'$. 

Actually the MMP terminates over the complement of a codimension $\ge 2$ subset $S'\subset Z'$: indeed, let $D$ be a prime divisor on $Z'$; then over $\eta_D$, we have  $P_Y\sim_\Q Q_Y\ge 0$ where $Q_Y$ is very exceptional over $\eta_D$ meaning $Q_Y$ does not contain some component of the pullback of $D$ that maps onto $D$; but then the MMP terminates over $\eta_D$ by \cite[Theorem 1.8]{B12} and  $K_Y+\Gamma_Y\sim_\Q 0$ over $\eta_D$. Therefore, $K_Y+\Gamma_Y\sim_\Q 0$ over $Z'\setminus S'$ for some codimension $\ge 2$ subset $S'\subset Z'$.\\ 

\emph{Step 3.}
Over $Z'\setminus S'$  we have the adjunction formula 
$$
K_Y+\Gamma_Y\sim_\Q g^*(K_{Z'}+\Gamma_{Z'}+M_{Z'})
$$
where $g$ denotes $Y\to Z'$, $\Gamma_{Z'}$ is the discriminant divisor and $M_{Z'}$ is the moduli divisor. Since the pullbacks of $K_X+B$ and $K_Y+\Gamma_Y$ to a common resolution of $X,Y$ agree over $U'$, $M_{Z'}$ is the same as the moduli part in Step 1 induced by $(X,B)\to Z$ \cite[3.4(2)]{B19}, and the pushforward of $\Gamma_{Z'}$ to $Z$ is just $B_Z$. 

Now consider 
$$
I_Y:=g^*(K_{Z'}+\Gamma_{Z'}+M_{Z'})-(K_Y+\Gamma_Y).
$$
By the above discussions, $I_Y\sim_\Q 0$ on $g^{-1}(Z'\setminus S')$, so $I_Y\sim_\Q J_Y$ where $\Supp J_Y$ maps into $S'$. In particular, $J_Y$ is very exceptional over $Z'$. Moreover, 
$$
-J_Y\sim_\Q -I_Y\sim_\Q K_Y+\Gamma_Y/Z',
$$
so $-J_Y$ is a limit of movable$/Z'$ $\Q$-divisors which shows $-J_Y\cdot C\ge 0$ for the very 
general curves of $E\to Z'$ for any component $E$ of $J_Y$, hence $J_Y$ is effective by the general negativity lemma (cf. \cite[Lemma 3.3]{B12}). 

We then deduce that 
$$
K_Y+\Gamma_Y+J_Y\sim_\Q g^*(K_{Z'}+\Gamma_{Z'}+M_{Z'})
$$
where $J_Y$ is effective.\\ 

\emph{Step 4.}
Let $Z''$ be a common resolution of $Z,Z'$. 
Let $\alpha\colon T\to X$ and $\beta\colon T\to Y$ be a common resolution so that the induced maps $T\bir W$ and $T\bir Z''$ are morphisms. 
Let
$$
G:=\beta^*(K_Y+\Gamma_Y+J_Y)-\alpha^*(K_X+B).
$$  
Then 
$$
\beta_*G=K_Y+\Gamma_Y+J_Y-\beta_*\alpha^*(K_X+B)
$$
$$
=K_Y+\Gamma_Y+J_Y-\psi_*(K_W+B_W) 
$$
$$
=\Gamma_Y+J_Y-\psi_*B_W\ge \psi_*\Gamma_W-\psi_*B_W \ge 0
$$ 
where $\psi$ denotes $W\bir Y$.
Thus, by the negativity lemma, $G\ge 0$ as $G$ is anti-nef over $Y$. 

Now since both $T\to Z'$ and $T\to Z$ factor through $Z''$ and since  
$$
\beta^*(K_Y+\Gamma_Y+J_Y)\sim_\Q 0/Z'
$$
and
$$
\alpha^*(K_X+B)\sim_\Q 0/Z,
$$
we have $G\sim_\Q 0/Z''$. Thus $G$ is the pullback of a divisor $E$ on $Z''$. By Step 2, 
over $U'$, we have 
$$
\beta^*(K_Y+\Gamma_Y)=\alpha^*(K_X+B),
$$ 
so over $U'$, $G=\beta^*J_Y$. By Step 3, $\Supp J_Y$ maps into $Z'\setminus S'$. This implies that $E$ is exceptional over $Z'$ which in turn implies that it is exceptional over $Z$ 
as $Z\bir Z'$ does not contract divisors. 

Therefore, 
$h_*\mathcal{O}_T(lG)\simeq \mathcal{O}_Z$ for every $l\in \N$ where $h$ denotes $T\to Z$. 
This in turn implies that we have an equality of algebras 
$$
R(q(K_Y+\Gamma_Y+J_Y))=R(q(K_X+B))
$$
because 
$$
lG\le {\rm Fix}|l\beta ^*(K_Y+\Gamma _Y+J_Y)|
$$ 
for any $l\in \N$ divisible by $q$.\\

\emph{Step 5.}
By Lemma \ref{l-ldisc}, the log discrepancies $a(D,X,B)\le 1$ belong to a fixed finite set depending only on $d,\Phi,u,v$. Thus the non-negative coefficients of $B_W$ belong to a fixed finite set, hence replacing $q$, we can assume that $q\Gamma_W=qB_W^{\ge 0}$ is integral. 
In particular, no coefficient of $\Gamma_W$ belongs to $(1-\frac{1}{q},1)$. 

Let 
$$
\Delta_Y:=\Gamma_Y-\frac{1}{q}\lfloor{\Gamma_Y}\rfloor +L_Y
$$ 
where $qL_Y$ is the pullback of a general element of $|3qdH_{Z'}|$. We claim that $(Y,\Delta_Y)$ is $\frac{1}{q}-lc$. Let $D$ be a prime divisor over $Y$. If $D$ is not exceptional over $Y$, then 
$a(D,Y,\Delta_Y)\ge \frac{1}{q}$ by definition of $\Delta_Y$. Assume $D$ is exceptional over $Y$. Then we can assume the centre of $D$ is not contained in $\Supp L_Y$, so 
$$
a(D,Y,\Delta_Y)\ge a(D,Y,\Gamma_Y)\ge a(D,W,\Gamma_W).
$$ 
If $a(D,W,\Gamma_W)>0$, then $a(D,W,\Gamma_W)\ge \frac{1}{q}$ as $q(K_W+\Gamma_W)$ is Cartier, so $a(D,X,\Delta_Y)\ge \frac{1}{q}$. So assume $a(D,W,\Gamma_W)=0$. Since $K_W+B_W$ is semi-ample and since $P$ (of Step 2) does not contain any non-klt centre of $(W,\Gamma_W)$, $W\bir Y$ is an isomorphism near the generic point of each non-klt centre of $(W,\Gamma_W)$. Then $W\bir Y$ is an isomorphism near the centre of $D$, hence again $a(D,X,\Delta_Y)\ge \frac{1}{q}$ because then $(Y,\Delta_Y)$ is klt log smooth near the centre of $D$ and  $q(K_Y+\Delta_Y)$ is Cartier.\\ 

\emph{Step 6.}
We claim that $\Ivol(K_Y+\Delta_Y)$ is bounded from above. Indeed 
$$
\Ivol(K_Y+\Delta_Y)\le \Ivol(K_Y+\Gamma_Y+L_Y)
$$
$$
\le \Ivol(K_Y+\Gamma_Y+J_Y+L_Y)
$$
$$
=\vol(K_{Z'}+\Gamma_{Z'}+M_{Z'}+3dH_{Z'})
$$
$$
\le \vol((1+3md)(K_{Z'}+\Gamma_{Z'}+M_{Z'}))
$$
$$
\le \vol((1+3md)(K_Z+B_Z+M_Z))
$$
$$
=(1+3md)^{\dim Z}v
$$
where $m$ is as in Step 1. The last inequality is actually equality by the isomorphism at the end of Step 4. \\

\emph{Step 7.}
Since $(X,B)$ is assumed klt over $\eta_Z$, $(Y,\Gamma_Y)$ is klt over $\eta_{Z'}$, so $\lfloor{\Gamma_Y}\rfloor$ is vertical over $Z'$. Thus over $\eta_{Z'}$, 
$$
K_Y+\Delta_Y=K_Y+\Gamma_Y-\frac{1}{q}\lfloor{\Gamma_Y}\rfloor +L_Y=K_Y+\Gamma_Y\sim_\Q 0.
$$ 
So $(Y,\Delta_Y)$ has a good minimal model $(Y',\Delta_{Y'})$ over $Z'$ by \cite[Theorem 1.5]{B12} and \cite{HX13}. By our choice of $L_Y$ and by boundedness of the length of extremal rays, $(Y',\Delta_{Y'})$ is globally a good minimal model of $(Y,\Delta_Y)$.  Note that by construction, $(Y',\Delta_{Y'})$ is $\frac{1}{q}$-lc, and $Y\bir Y'$ is an isomorphism over $\eta_{Z'}$, so $(Y',\Delta_{Y'})$ is birationally a crepant model of $(X,B)$ over $\eta_Z$ because $(Y,\Gamma_Y)$ is so.

Let $Y'\to \bar{Z}/Z'$ be the contraction defined by $K_{Y'}+\Delta_{Y'}$.
Let $A_W$ be the pullback of $A$ to $W$ and let $A_{Y'}$ be the pushforward to $Y'$ of the horizontal$/\bar{Z}$ part of $A_W$ (note that $A$ may not be $\Q$-Cartier globally but it is $\Q$-Cartier over $\eta_Z$ and we only need the pullback of $A$ over $\eta_Z$ in order to define $A_{Y'}$). Then over $\eta_{\bar{Z}}$, $A_{Y'}$ is the pullback of $A$ under $Y'\bir X$. By construction, over $\eta_Z$, $X\bir Y'$ does not contract any divisor which means that, over $\eta_Z$, $X$ is the ample model of $A_{Y'}$. Thus replacing $Y'$ with the ample model of $A_{Y'}$ over $\bar{Z}$, $X\bir Y'$ becomes an isomorphism over $\eta_{\bar{Z}}$. In particular, now $A_{Y'}$ is an integral divisor because $A$ is an integral divisor by Step 1.\\ 

\emph{Step 8.}
By construction, $(Y',\Delta_{Y'})$ is $\frac{1}{q}$-lc, $q\Delta_{Y'}$ is integral, $A_{Y'}$ is integral and ample over $\bar{Z}$, the volume of $A_{Y'}$ on the general fibres of $Y'\to \bar{Z}$ is bounded, and the Iitaka volume of $K_{Y'}+\Delta_{Y'}$ is bounded. Therefore, 
by Lemma \ref{l-bnd-klt-fibs}, $(Y',\Delta_{Y'})$ is crepant birationally bounded in the sense that there is a bounded projective pair $(Y'',\Delta_{Y''})$ over $\bar{Z}$ where $Y''\bir Y'/\bar{Z}$ is a birational map whose inverse does not contract any divisor, and $K_{Y''}+\Delta_{Y''}$ is the pullback of $K_Y+\Delta_Y$. In particular, there exists a very ample divisor $R_{Y''}$ so that 
$\vol(K_{Y''}+\Delta_{Y''}+R_{Y''})$ is bounded which implies that 
$\vol(R_{Y''}+H_{Y''})$ is bounded where $H_{Y''}$ is the pullback of $H_{Z'}$, because of our choice of $L_Y\le \Delta_Y$.

Let $K_{Y'}+\Gamma_{Y'}+J_{Y'}$ be the pushforward of $K_{Y}+\Gamma_{Y}+J_{Y}$. Since 
$$
K_{Y}+\Gamma_{Y}+J_{Y}\sim_\Q g^*(K_{Z'}+\Gamma_{Z'}+M_{Z'}),
$$
$K_{Y'}+\Gamma_{Y'}+J_{Y'}$ is $\Q$-Cartier. 
Let $K_{Y''}+\Omega_{Y''}$ be the pullback of $K_{Y'}+\Gamma_{Y'}+J_{Y'}$. Then from $\Gamma_{Y'}+L_{Y'}\ge \Delta_{Y'}$ we get $\Omega_{Y''}+L_{Y''}\ge \Delta_{Y''}$, where $L_{Y''}$ is the pullback of $L_{Y'}$.

By Step 1, 
$$
-(K_{Z'}+B_{Z'}+M_{Z'})+nH_{Z'}
$$ 
is big for some fixed $n$. 
Note since the pushforwards of $\Gamma_{Z'}$ and $B_{Z'}$ to $Z$ are both equal to $B_Z$ and $B_{Z'}$ contains all the exceptional divisors of $Z'\bir Z$ with coefficient one, it follows that $\Gamma_{Z'}\le B_{Z'}$. 
This in turn implies that
$$
-(K_{Z'}+\Gamma_{Z'}+M_{Z'})+nH_{Z'}
$$ 
is big.  Therefore, 
$$
(K_{Y''}+\Omega_{Y''})\cdot R_{Y''}^{d-1}\le nH_{Y''}\cdot R_{Y''}^{d-1}\le n\vol(H_{Y''}+R_{Y''})
$$ 
is bounded, hence $\Omega_{Y''}\cdot R_{Y''}^{d-1}$ is bounded.\\ 

\emph{Step 9.}
Recall $T$ from Step 4. We can assume it is a common resolution of $W,Y',Y''$. 
Let $K_T+B_T$ be the pullback of $K_X+B$ and let $B_{Y''}$ be the pushforward of $B_T$. Then applying the negativity lemma over $Y'$ or just using Step 4, 
$$
\theta^*(K_{Y'}+\Gamma_{Y'}+J_{Y'})-(K_T+B_T)\ge 0
$$ 
where $\theta$ denotes $T\to Y'$. So using the fact that 
$K_{Y''}+\Omega_{Y''}$ is the crepant pullback of $K_{Y'}+\Gamma_{Y'}+J_{Y'}$, we have  
$$
K_{Y''}+\Omega_{Y''}\ge K_{Y''}+B_{Y''}.
$$ 
Thus if $\Gamma_{Y''}:=B_{Y''}^{\ge 0}$, then $\Omega_{Y''}\ge \Gamma_{Y''}$, hence $\Gamma_{Y''}\cdot R_{Y''}^{d-1}$ is bounded by Step 8.

Note that the coefficients of $\Gamma_{Y''}$ belong to a fixed finite set as in Step 5, so $(Y'',\Gamma_{Y''})$ belongs to a bounded family. So there is a log resolution $V\to Y''$ so that if $\Theta$ is the support of the birational transform of $\Gamma_{Y''}$ union the reduced exceptional divisor, then $(Y,\Theta)$ is log smooth and bounded.\\

\emph{Step 10.}
Replacing $T$ with a higher resolution, we can assume $T\bir V$ is a morphism. Let $\Gamma_V=B_{V}^{\ge 0}$ where $B_V$ is the pushforward of $B_T$. Then 
$\Supp \Gamma_V\subseteq \Supp \Theta$, so $(V,\Gamma_V)$ belongs to a bounded family. Moreover, the coefficients of $\Gamma_V$ belong to a fixed finite set.

Now  
$$
R(q(K_X+B))=R(q(K_T+B_T))=R(q(K_T+B_T^{\ge 0}))\subseteq R(q(K_V+\Gamma_V)).
$$ 
Moreover, $\Gamma_{Y'}$ is just the pushforward of $\Gamma_V$ to $Y'$.
Then 
$$
R(q(K_V+\Gamma_V))\subseteq R(q(K_{Y'}+\Gamma_{Y'}))
$$
$$
\subseteq R(q(K_{Y'}+\Gamma_{Y'}+J_{Y'}))
=R(q(K_{Y}+\Gamma_{Y}+J_{Y}))
$$
$$
=R(q(K_X+B))
$$ 
where the last equality is by Step 5 while the preceding equality follows from the fact that $$
K_{Y}+\Gamma_{Y}+J_{Y}\sim_\Q 0/Z'.
$$

Therefore, 
$$
R(q(K_V+\Gamma_V))=R(q(K_X+B))
$$ 
meaning that $Z$ is the lc model of $(V,\Gamma_V)$. But then since $(V,\Gamma_V)$ is log smooth and it belongs to a bounded family, its lc model also belongs to a bounded family, by \cite[Theorem 1.2]{HMX18}. More precisely, there is a very ample divisor $H_Z$ so that $H_Z^{\dim Z}$ and  $(K_Z+B_Z+M_Z)\cdot H_Z^{\dim Z-1}$ are bounded.

\end{proof}


\bigskip

Yau Mathematical Sciences Center, Jing Zhai, Tsinghua University, Hai Dian District, Beijing, China 100084; email: birkar@tsinghua.edu.cn\\

Department of Mathematics, University of Utah,
155 South 1400 East, JWB 233, Salt Lake City, UT 84112, USA; e-mail: hacon@math.utah.edu

\end{document}